\documentclass[12pt,a4paper]{article}
\usepackage[utf8]{inputenc}
\usepackage{amsmath}
\usepackage{amsfonts}
\usepackage{amssymb}
\usepackage{amsthm}
\usepackage{color}

\def\ga{\alpha}\def\gb{\beta}
\def\gga{\gamma}\def\gG{\Gamma}
\def\gd{\delta}\def\gD{\Delta}
\def\gep{\varepsilon}

\def\gs{\sigma}
\def\gp{\varphi}\def\gP{\Phi}
\def\gO{\Omega}

\def\cB{\mathcal{B}}\def\cF{\mathcal{F}}

\def\cH{\mathcal{H}}\def\cI{\mathcal{I}}
\def\cK{\mathcal{K}}

\def\cS{\mathcal{S}}\def\cD{\mathcal{D}}
\def\cA{\mathcal{A}}

\def\cT{\mathcal{T}}
\def\cZ{\mathcal{Z}}

\def\bA{\mathbb{A}}\def\bB{\mathbb{B}}
\def\bN{\mathbb{N}}
\def\bR{\mathbb{R}}\def\bS{\mathbb{S}}
\def\bX{\mathbb{X}}

\def\bE{\mathbb{E}}
\def\bP{\mathbb{P}}
\def\bY{\mathbb{Y}}
\def\abs#1{\vert #1 \vert} 

\theoremstyle{definition}
\newtheorem{defn}{Definition}[section]

\theoremstyle{plain}
\newtheorem{thm}[defn]{Theorem}
\newtheorem{prop}[defn]{Proposition}
\newtheorem{lem}[defn]{Lemma}
\newtheorem{cor}[defn]{Corollary}
\theoremstyle{remark}
\newtheorem{rk}[defn]{Remark}

\usepackage{fullpage} 
\usepackage{shuffle}
\usepackage{hyperref}
\numberwithin{equation}{section}
\usepackage[english]{babel}

\usepackage{tikz} 
\usepackage{tikz-cd}
\usetikzlibrary{trees}
\usepackage{difftrees}
\usepackage{enumitem}
\usepackage{bbm}

\def\abs#1{\vert #1 \vert}

\def\t{\tau}

\def\1{\mathbf{1}}
\def\X{\mathbf{X}}

\def\qsT{\hat{T}}
\def\qshuffle{\,\widehat{\shuffle}\,}

\makeatletter
\newsavebox{\@brx}
\newcommand{\llangle}[1][]{\savebox{\@brx}{\(\m@th{#1\langle}\)}%
  \mathopen{\copy\@brx\mkern2mu\kern-0.9\wd\@brx\usebox{\@brx}}}
\newcommand{\rrangle}[1][]{\savebox{\@brx}{\(\m@th{#1\rangle}\)}%
  \mathclose{\copy\@brx\mkern2mu\kern-0.9\wd\@brx\usebox{\@brx}}}
\newcommand{\llbracket}[1][]{\savebox{\@brx}{\(\m@th{#1[}\)}%
  \mathopen{\copy\@brx\mkern2mu\kern-0.9\wd\@brx\usebox{\@brx}}}
\newcommand{\rrbracket}[1][]{\savebox{\@brx}{\(\m@th{#1]}\)}%
  \mathclose{\copy\@brx\mkern2mu\kern-0.9\wd\@brx\usebox{\@brx}}}
\makeatother

\title{Quasi-geometric rough paths and \\rough change of variable formula}

\author{Carlo Bellingeri \thanks{TU Berlin, \texttt{bellinge@math.tu-berlin.de}}}%
\date{}

\begin{document}
\maketitle
\begin{abstract}
Using some basic notions from the theory of Hopf algebras and quasi-shuffle algebras, we introduce rigorously a new family of rough paths: the quasi-geometric rough paths. We discuss their main properties. In particular, we will relate them with iterated Brownian integrals and the concept of ``simple bracket extension'', developed in the PhD thesis of David Kelly \cite{kelly2012ito}. As a consequence of these results, we have a sufficient criterion to show for any $\gga\in (0,1)$ and any sufficiently smooth function $\gp\colon \bR^d\to \bR$ a rough change of variable formula on any $\gga$-H\"older continuous path $x\colon [0, T]\to \bR^d$, i.e. an explicit expression of $\gp(x_t)$ in terms of rough integrals.
\end{abstract}
\tableofcontents
\section{Introduction}
Since it began with the seminal paper \cite{Chen54} by Kuo-Tsai Chen, the theory of iterated integrals for smooth paths has shown a rich algebraic structure, pointing out deep connections between algebraic geometry and differential equations, see \cite{Cartier88}. In more recent years, the corresponding theory of iterated stochastic integrals has led to the appearance of new algebraic structures, enriching the previous framework with additional connections with combinatorics and stochastic analysis, see \cite{Rota_Wallstrom97}. 

We can trace in \cite{gaines94} a simple example to describe how stochastic integration augments the standard algebraic structures of smooth paths. We will recall it using a slightly different notation. Given a $\bR^d$ Brownian motion $B=(B^1,\cdots , B^d)$, Gaines  considered  for any $t>0$ and any word $w=i_1\cdots i_n$ for some  $ i_1, \cdots, i_n\in \{1\,,\cdots,d \} $ the random variables
\begin{equation}\label{def_I_and_S}
 S_t(w)= \int_{\gD^n_{0t}}d^{\circ }B^{i_1}_{s_1}\cdots d^{\circ}B^{i_n}_{s_n}\,,\quad I_t(w)= \int_{\gD^n_{0t}}dB^{i_1}_{s_1}\cdots dB^{i_n}_{s_n}\,,
\end{equation}
where $\gD^n_{0t}= \{0\leq s_1<\cdots <s_n\leq t\}$ and  $dB$, $d^{\circ} B$  denote respectively the It\^o-Wiener and  Stratonovich integral. A practical tool introduced in that paper to express the relations induced by these variables can be obtained by introducing two specific algebraic operations.

We start from the finite set $\{1\,,\cdots \,, d,\circ \}$ and  we consider $V^{\circ}$ the free real vector space generated the set of words constructed from $\{1\,,\cdots \,, d,\circ \}$ union the empty word $\1$. The vector space $V^{\circ}$ admits a first natural commutative operation $\shuffle$ over it. We define it recursively from the identities $\1 \,\shuffle\, v= v\, \shuffle\,\1 =v$ for any $v\in V^{\circ} $ and for any couple of words $v,w$ and letters $a,b\in \{1\,,\cdots \,, d,\circ \}$ we impose the relation
\begin{equation}\label{defn_shuffle}
va \shuffle wb= (v \shuffle wb)a + (va \shuffle w)b\,.
\end{equation}
The resulting operation is called \emph{shuffle product} and the resulting structure $(V^{\circ}, \shuffle)$ is one example of \emph{shuffle algebra}, a standard object in modern commutative algebra (see \cite{reutenauer1993free} for further properties). In addition to this,  a second operation  $ \diamond$, called \emph{It\^o product}, has been introduced in  \cite{gaines94}. To obtain it, we keep the same relations with $\1$ and we slightly modify   \eqref{defn_shuffle} as follows
\begin{equation}\label{defn_diamond}
ua\diamond vb:= (u\diamond vb)a+ (ua\diamond v)b+ (\mathbbm{1}_{\{1\,,\cdots \,, d\}}(a)\delta_{ab}) (u\diamond v)\circ\,,
\end{equation}
where $\gd_{ab}$ is the usual Kronecker's delta. The reason behind the definition of these products is simply motivated by an interesting property. Looking at the random values $S_t(\cdot)$, $I_t(\cdot)$ from \eqref{def_I_and_S} as linear functional from $V^{\circ}$ to $\bR$, where we replace every occasion of the letter $\circ$ with a Lebesgue integral against $ds$,  one has from \cite[Prop 2.2, Prop 2.3]{gaines94} that these functions are two \emph{real characters} over the two commutative algebras $(V^{\circ}, \shuffle)$ and $(V^{\circ}, \diamond)$ respectively, meaning that for any $w, v\in V^{\circ}$ one has the a.s. identities  
\begin{equation}\label{charachter_brownian}
S_t(v\shuffle w)= S_t(v )S_t( w) \,, \quad I_t(v\diamond w)= I_t(v )I_t( w)\,.
\end{equation}
Recalling from \cite{Chen54} that the iterated integral of smooth paths are naturally character with respect to the shuffle product, the nature of Brownian integration is intimately associated with a different type of algebraic structure.

The It\^o product is the first example of a \emph{quasi-shuffle product}, a notion which was not formulated rigorously when \cite{gaines94} was published. Introduced in the seminal paper \cite{hoffman2000} and extended in \cite{hoffman17}, the theory of \emph{quasi-shuffle algebras} was designed to describe the intrinsic relations between the multiple-zeta vlaues. However, many examples from stochastic analysis, like Levy processes and semi-martingales, satisfy naturally quasi-shuffle relations, see \cite{curry14,kurusch15, Ebrahimi-Fard2015}.

Iterated integrals and stochastic processes are not simply related via quasi-shuffle algebras. Introduced in \cite{lyons1998}, the theory of \emph{rough paths} was built to formulate rigorously a differential equation of the form
\begin{equation}\label{rough_differential_equation}
dY_t=  \sum_{i=1}^d f_i(Y_t)dx^i_t\,\quad Y_0=y_0\in \bR^e
\end{equation}
where the source $x=(x^1,\cdots, x^d)\colon [0,T]\to \bR^d$ is an generic irregular path (e.g. Brownian trajectories) and  $f_i\colon \bR^e\to \bR^e$ are smooth vector fields.  The knowledge of the path $x$ is not enough to establish  existence nor the uniqueness of \eqref{rough_differential_equation}. To fill this lack of well-posedness, we can replace $x$ with a \emph{geometric rough path} $X$, see Definition \ref{defn_geom_rough} below, a finite family of functions containing the increments of $x$ and other non-differentiable functions of it. Moreover equation \eqref{rough_differential_equation} is rewritten in terms of a \emph{rough differential equation}, an integral equation concerning an explicit notion of integration against a rough path $X$. The main properties satisfied by these additional functions are shuffle-type identities, the same algebraic relations of smooth iterated integrals, thereby extending Chen's theory to a non-smooth setting. Even if these notions are purely deterministic in their formulation, the theory of stochastic processes provides many explicit examples of geometric rough paths structures, enriching the tools of stochastic analysis, see \cite{Friz2020course} for a general introduction.

Several extensions of geometric rough paths were presented in the literature, to take into account more general rough differential equations. We recall in particular \emph{branched rough paths}, see \cite{gub10,hairer2015geometric} and  Definition \ref{defn_branched} below, where the key features of geometric rough paths are extended to the Butcher-Connes-Kreimer Hopf algebra of rooted forests. We mention also the family of \emph{planarly branched rough paths}, see \cite{Manchon2018}, which are used to formulate rough differential equations on a homogeneous manifold. However, except for an oral talk by David Kelly at the ICMAT in Madrid, held in 2013, no results between quasi-shuffle algebras and rough paths have been studied, despite the natural quasi-shuffle nature of iterated It\^o-Wiener Brownian integrals.

In this paper, we introduce rigorously the notion of \emph{quasi-geometric rough paths}\footnote{This name was given by David Kelly in the aforementioned talk.} and their main properties. Particularly, we will use them to prove the following results:
\begin{itemize}
\item Using quasi-geometric Lyons' extension theorem in Proposition \ref{ext_lyo_quasi}, we obtain an alternative proof of the correction formula between iterated It\^o-Wiener Brownian integrals and iterated Stratonovich Brownian integrals, as expressed in Corollary \ref{first_main_thm}.
\item For any generic $\gga$-H\"older continuous path $x\colon [0, T]\to \bR^d$ which is associated to branched rough path, we obtain in Theorem \ref{big_thm} a sufficient condition for a rough change of variable formula for branched rough paths. Moreover, in Theorem \ref{quasi_shuffle_change_thm} we obtain another rough change of variable formula when $x$ is associated to a  \emph{quasi-geometric bracket extension}, see Definition \ref{Def_quasi_geom_bracket}  below.
\end{itemize}
Whereas the first result provides alternative reasoning to describe the classical relations between $I_t(\cdot)$ $S_t(\cdot)$ given in  \eqref{def_I_and_S}, the second result deserves a deeper explanation.

Let $\gp\colon \bR^d\to \bR$ be a smooth function and $x\colon [0,T]\to \bR^d$, $x=(x^1,\cdots, x^d)$ a $C^1$ path. We call \emph{classical change of variable formula} the integral version of the chain rule,  i.e. the  identity
\begin{equation}\label{classic_change}
\gp(x(t))= \gp(x(s))+ \sum_{i=1}^d\int_s^t \partial_{x_i}\gp(x(r))\dot{x}^i(r)dr\,.
\end{equation}
This formula is a cornerstone of standard calculus, holding if and only if $x$ is absolutely continuous function. In particular, when $x$ does not satisfy this property, the integral in \eqref{classic_change} might not be well defined because $z$ is not a.e. differentiable and Lebesgue integration theory is not useful any more. Surprisingly, it is still possible to write a generalised change of variable formulae, provided that we radically change our notion of integration. A remarkable example where this phenomenon happens comes from stochastic analysis when $x= \{x_t\}_{t\geq 0}$ is a realisation of a continuous semi-martingale, whose trajectories are a.s. not a bounded variation function.  One of the big achievements of stochastic calculus was to show that even in this case there is a new change of variable formula for $x$, the well celebrated \emph{It\^o formula} (see e.g. \cite[Chap. 4]{revuz2004continuous})
\begin{equation}\label{ito_classic_change}
\gp(x_t)= \gp(x_s)+ \sum_{i=1}^d\int_s^t \partial_{x_i}\gp(x_r)dz^i_r+ \frac{1}{2} \sum_{i,j=1}^d\int_s^t \partial_{x_ix_j}\gp(x_r)d[x^i, x^j]_r\,.
\end{equation} 
To achieve this new identity,  it is necessary to develop the whole theory of stochastic integration, where the path-wise nature of the integration is replaced with the probabilistic nature of $x$, and the resulting formula looks different because it involves a term with the second-order derivatives of $\gp$. 
 
Motivated by this fundamental result, we explore in the rough change of variable formula the same type of problem when $x$ is a generic $\gga$-H\"older path which is associated to a branched rough path $X$. The choice of a branched rough path is motivated by the absence of apriori relationships between the products of the coordinates of $x$ and the additional components of $X$. A first general theory which expresses the rough change of variable formula and a more general identity on branched rough differential equations was given in the last chapter of David Kelly's PhD Thesis \cite[Chap. 5]{kelly2012ito}. In its formulation Kelly introduced the key-notion of \emph{simple bracket extension}, see Definition \ref{bracket}.  This condition allows to obtain an extremely general formula Theorem \ref{chg_thm} but at the same time, this notion requires to validate some additional properties, which make this definition more arduous for applications. On the other hand, we  propose the concept of \emph{quasi-geometric bracket extension} see Definition \ref{Def_quasi_geom_bracket}, which we can explicitly link to simple bracket extension in Theorem \ref{big_thm} and they imply themselves a rough change of variable formula with quasi-geometric terms, as explained in Theorem \ref{quasi_shuffle_change_thm}.

The present paper is organised as follows. We first recall in Section $2$  the main properties of rough paths theory using Hopf algebras to have a general framework which includes shuffle, quasi-shuffle and branched structures in a unified setting. Then we introduce in Section $3$ the notion of quasi-geometric rough paths and its link with the correction formula in Corollary \ref{first_main_thm}. We conclude the paper with a detailed study of the rough change of variable formula, by recalling the main ideas behind Kelly's theory and the general properties of quasi-geometric bracket extensions. Most of the presented results are deterministic but we can apply them to the realisation of any a.s. H\"older continuous stochastic process. We remark that a recent paper \cite{Cont2019} studies the same problem of rough change of variable formulae, establishing some connections with F\"ollmer calculus \cite{follmer81}. 
\subsection*{Acknowledgements}
The author was supported by DFG Research Unit FOR2402. Moreover, the author is very grateful to Lorenzo Zambotti, Nikolas Tapia, Charles Curry, Frédéric Patras and Kurusch Ebrahimi-Fard for many suggestions concerning the theory of rough paths and quasi-shuffle algebras. This project was conceived during the research week the author spent at NTNU in Trondheim University, which received support from Campus France, PHC Aurora 40946NM. The author wrote some parts of this paper while he was doing his PhD at Sorbonne Université, working at the Laboratoire de Probabilités, Statistique et Modélisation, UMR 8001.

\section{Rough paths and Hopf algebras}\label{section1}
To have a synthetic description of different families of rough paths, we recall the general definitions and construction of  rough paths, see \cite{Tapia2018, Manchon2018}, when they take value over a generic commutative, connected and graded Hopf algebra $\cH$. 
\subsection{Algebraic preliminaries}\label{algebra}
For an introduction on Hopf Algebras we refer to \cite{Manchon2008,sorindascalescu2000}. Loosely speaking, a commutative Hopf algebra is just a specific type of commutative bialgebra, i.e. a vector space $\cH$ over a field $k$, which we will always assume $\bR$ in our context, endowed with two fundamental operations: a commutative and associative \emph{product} $m\colon \cH\otimes \cH\to \cH$ (whose operation will be denoted as a juxtaposition) and a linear map $\Delta\colon  \cH\to \cH\otimes \cH $ called \emph{coproduct} which satisfies the coassociativity identity
\begin{equation}\label{eq:coassociativity}
(\text{id}\otimes \Delta)\Delta= (\Delta\otimes \text{id})\Delta\,
\end{equation}
and  some  compatibility relations between $m$ and $\Delta$. The main consequence of this conditions is that we can use the coproduct $\Delta$ to define the \emph{convolution product} $*\colon \cH^*\otimes \cH^*\to \cH^*$ on $\cH^*$, the algebraic dual of $\cH$, trough the identity
\begin{equation}\label{eq:convolution_product}
\langle \alpha*\beta, h \rangle:=\langle\alpha \otimes \beta, \Delta h \rangle\,,
\end{equation}
where $\langle \cdot , \cdot\rangle $ is the canonical pairing between $\cH^*$ and $\cH$. As always we suppose that there exists a unity $\1\in \cH$ and a \emph{counity} $ \1^*\in \cH^{*}$, which satisfies the identity 
\begin{equation}\label{eq:counity}
(\text{id}\otimes \1^*)\Delta= (\1^*\otimes \text{id})\Delta= \text{id}\,. 
\end{equation}
The notation for the counity is chosen because we can fix $\1^*$ as the only element of $(\cH)^*$ satisfying $\1^*(\1)=1$ and zero elsewhere. Using  identity \eqref{eq:counity} and the choice of the counity, $\Delta$ can be always written as
\begin{equation}\label{defn_coproduct}
\gD h= \1\otimes h+ h\otimes \1+ \gD'h\,.
\end{equation}
for some application $\gD'\colon \text{Ker}(\1^*)\to \text{Ker}(\1^*)\otimes \text{Ker}(\1^*)$. We call it \emph{reduced coproduct}. Following the definition of $\gD'$ and the properties of $\gD$, it is trivial to show that the reduces coproduct satisfies also 
\begin{equation}\label{eq:coassociative_reduced_coprod}
(\text{id}\otimes \Delta')\Delta'= (\Delta'\otimes \text{id})\Delta'\,.
\end{equation}
We call a \emph{character} on $\cH$ every non-zero linear map $X\colon\cH\to\bR$ such that
\[ \langle X, hk\rangle = \langle X,h\rangle\langle X,k\rangle\,. \]
for all $h,k\in\cH$. We denote by $G$ the set of all characters on $\cH$. Generally the triplet $(G,*,\1^*)$ is only a semi-group when  $\cH$ is a bialgebra. The main property in the definition of a Hopf algebra is the existence of an additional linear map $\cA\colon \cH\to \cH$ called \emph{antipode} satisfying the identity
\[m(\text{id}\otimes \cA)\Delta h= m(\cA\otimes \text{id})\Delta h= \1^*(h) \1\,,\]
for every $h\in \cH$. Using its properties, it is a straightforward result to show that the antipode is unique and $G$ is a group whose inverse is given by  $g^{-1}=g\circ \cA$ (see \cite[Remark 4.2.3]{sorindascalescu2000}).

The Hopf algebra we will consider are also \emph{locally finite graded} and \emph{connected}, that is  we can decompose them as the direct sum of vector spaces
\begin{equation}\label{eq:grading}
\cH= \bigoplus_{n=0}^{\infty} \cH_n 
\end{equation}
where $\cH_0=\langle\1\rangle$ is the  linear space generated by $\1$ and every vector space $\cH_n$ is finite dimensional. The operations $m$ and $\Delta$ and $\cA$ are also compatible with the grading, meaning that they satisfy the properties
\begin{equation}\label{eq_compatibility}
m\colon \cH_n\otimes \cH_m\to \cH_{n+m}\,,\quad \Delta\colon \cH_n\to \bigoplus_{\substack{p+q=n\\ p,q\geq 0}}\cH_p\otimes \cH_q\,,\quad  \cA(\cH_n)\subset \cH_n\,.
\end{equation}
A fundamental result in the algebraic literature (see again \cite[Chapter 4]{sorindascalescu2000}) shows that every  locally finite graded and connected bialgebra admits an antipode $\cA$, so in this case a Hopf Algebra is identified by describing  only the operations $m$ and $\Delta$. Moreover the resulting reduced coproduct $\gD'$ satisfies
\begin{equation}
\gD' \cH_n \subset \bigoplus_{\substack{p+q=n\\ p,q\geq 1}}\cH_p\otimes \cH_q\,.
\end{equation}
Since every subspace $\cH_n$ is finite dimensional, an equivalent procedure to encode the grading is the existence of a countable set $\cB\subset \cH$ containing $\1$ and a function $g \colon \cB\to \bN$ satisfying the following properties:
\begin{itemize}
\item[$1h)$] $g^{-1}(0)=\{\1\}$ and $g^{-1}(n)$ is finite for every $n\geq 1$; 
\item[$2h)$] for any $n\geq 0$ one has the identity
\begin{equation}\label{eq:grading_g}
\cH_n=\langle g^{-1}(n)\rangle. 
\end{equation}
\end{itemize}
We call the couple $(g, \cB)$ a \emph{grading for $\cH$}. The graded structure allows to define a "finite-dimensional" version of  the character group $G$. For any $N\geq 1$ we introduce the vector space
\begin{equation}\label{eq:grading2}
\cH^N= \bigoplus_{n=0}^{N} \cH_n
\end{equation}
Thanks to the properties of $\gD$ in \eqref{eq_compatibility}, the triple $(\cH^N, \gD, \1^*)$ is still a coalgebra and using the identity \eqref{eq:grading_g}, we deduce that the set  $\cB^N:= g^{-1}(\{0\,,\cdots\,, N\})$ is a basis of $ \cH^N$ for every $N\geq 1$. Dualising this property and using the hypothesis on $\cH$, the triple  $((\cH^N)^*, *, \1^*) $ becomes a finite dimensional graded algebra which can be canonically written as
\begin{equation}\label{eq:grading_dual}
(\cH^N)^*=\bigoplus_{n=0}^{N} (\cH_n)^*\,
\end{equation}
and the set $\cB^N$ is identified with the dual basis of $(\cH^N)^*$ (every element of it will be denoted by $u^*$ where $u\in \cB^N$). We say that $X\in(\cH^N)^*\setminus\{0\}$ is a \emph{truncated character} on $\cH_N$ if the identity
\begin{equation}\label{eq:charachter_truncated}
 \langle X,xy\rangle=\langle X,x\rangle\langle X,y\rangle 
 \end{equation}
holds for all $x\in\cH_{n},y\in\cH_{m}$ with $n+m\le N$. We call $G^N$ the set of truncated characters on $\cH_N$. Using the same operation of the convolution product, it turns out that $(G^N,*, \1^*)$ is still a group  and the inverse is  given using again the antipode $\cA$. 

There is a natural identification  of the group $G^N$  as a family of automorphism of the vector space $\cH_N$. Indeed for every $X\in G^N$ we  consider the linear map $\gG_{X}\colon \cH^N\to \cH^N$, defined by 
\begin{equation}\label{eq_operators}
\gG_X h:= (X^{-1}\otimes \text{id})\gD h\,.
\end{equation}
Using the properties of $\gD$ it is straightforward to show that these operators are invertible and the application $X\to \gG_X$ is an injective homomorphism. We will henceforth look at $G^N$ in both ways. We recall that $G^N$ is indeed a Lie Group whose Lie algebra has also an explicit description. Looking at the properties of $(\cH^N, G^N)$ we also deduce that this couple is also an example of a regularity structure (see \cite{Hairer2014} for the definition) as it was already pointed out in \cite[Lemma 4.18]{Hairer2014}.

\subsection{Weighted rough paths and controlled rough paths}
In what follows, we suppose given a locally finite graded and connected Hopf algebra $\cH$ where we fixed a grading $(g, \cB)$ for $\cH$, called \emph{natural grading}. We use the shorthand notation $\vert \cdot\vert$ to denote the function $g$ which we will call it \emph{natural weight}. The notion of a rough path wants to encode into a formal object all the \emph{increments} and some non linear functional of a $\gga$-H\"older path where $\gamma\in (0,1)$. 
\begin{defn}  \label{dfn:genRP}
Let $\gamma\in (0,1)$ and  $N_{\gamma}=\lfloor\gamma^{-1}\rfloor  $ be the biggest integer  $N\gamma\leq 1$. A \emph{$\gamma$-rough path over $\cH$} is a function $X\colon[0,T]^2\to (\cH^{N_{\gamma}})^*$ satisfying these three properties:
\begin{itemize}
\item for all $x\in\cH_{n}$, $y\in\cH_{m}$ with $n+m\le N_{\gamma}$ one has 
\begin{equation}\label{eq:charachter_truncated_rough}
\langle X_{st},xy\rangle=\langle X_{st},x\rangle\langle X_{st},y\rangle \,
\end{equation}
for any $s,t\in [0,T]$;
\item for any $s,u,t\in[0,T]$ one has the so called Chen's identity
\begin{equation}\label{eq:chen}
X_{su}* X_{ut}=X_{st}\,;
\end{equation}
\item for all $u\in \cB^{N_{\gamma}}$ 
\begin{equation}\label{eq:genrpbound}
\sup_{s\neq t\in [0,T]^2}\frac{\abs{\langle X_{st},u\rangle}}{\abs{t-s}^{\gga\abs{u}}}<+\infty\,.
\end{equation}
\end{itemize}
Denoting  by $\{e_i\}_{i=1,\ldots,d}$ the  basis of $\cH_1$ induced by $g$, if a path $x\colon[0,T]\to\bR^d$, $x= \{x^i\}_{i=1,\ldots,d}$ satisfies $x^i_t-x^i_s=\langle X_{st},e_i\rangle$, we say that $X$ is a rough path over $x$.
\end{defn}

\begin{rk}
The notion of $\gamma$-rough path extends the usual notion of  increments of a $\gamma$-H\"older path $x\colon[0,T]\to\bR^d$. Indeed every $e_i$ satisfies by construction of $\gD$ and $\gD'$
\[\gD e_i= \1\otimes e_i+ e_i\otimes \1\,.\]
Therefore the identity \eqref{eq:chen} yields
\[\langle X_{st},e_i\rangle= \langle X_{su},e_i\rangle+ \langle X_{ut},e_i\rangle\,\]
and there exists a unique $\gamma$-H\"older path $x'\colon[0,T]\to\bR^d$ satisfying $x'_0=0$ such that $X$ is over $x'$. More generally, a function $X\colon[0,T]^2\to \cH^{N_{\gamma}}$ is a $\gamma$-rough path over if and only if there exists a unique path $\bX\colon [0,T]\to G^{N_{\gamma}}$ satisfying $\bX_0= \1^*$ which is $\gamma$-H\"older with respect to an intrinsic metric defined on  $G^{N_{\gamma}}$ (see \cite[Prop 3.3]{Tapia2018}).
\end{rk}
To model a  rough path where every function $\{ \langle X_{st},u\rangle\}_{u\in  \cB^{N_{\gamma}}}$  is allowed to have different H\"older exponents, we introduce the following notion of compatible weight.
\begin{defn}
Let  $\omega\colon \cB\to \bN$ be a generic surjective function. We say that $\omega$ is \emph{compatible weight} for $\cH$ if for every $N\geq 1$ one has the inclusions  
\begin{equation}
\cB^N_{\omega}\subset \cB^N\,, \quad \Delta \cH^{N}_{\omega} \subset \cH^{N}_{\omega} \otimes \cH^{N}_{\omega},
\end{equation}
where $\cB^N_{\omega} := \omega^{-1}(\{0\,,\cdots\,, N\})$  and $\cH^{N}_{\omega} $ is the vector space generated by $\cB^N_{\omega}$.
\end{defn}
Using a standard terminology in algebra, If $\omega$ is a compatible weight then $(\cH^{N}_{\omega}, \gD, \1^*)$ is a \emph{subcoalgebra} of $(\cH^N, \gD, \1^*)$. We also write it as the direct sum
\[\cH^{N}_{\omega}=\bigoplus_{n=0}^{N} \cH_{n,\omega}^N\,, \quad \cH_{n,\omega}^N:=\cH_{n}\cap\cH^{N}_{\omega}\,.\]
Considering $(\cH^{N}_{\omega})^*$, the hypothesis of compatibility implies the dual algebra $((\cH^{N}_{\omega})^*, *, \1^*)$ is a subalgebra of $((\cH^N)^*, *, \1^*)$ but as in case of $\cH^{N}$, $\cH^{N}_{\omega}$ is not an algebra in general. The hypothesis of surjectivity for $\omega$ is imposed so that we recover also a graded coalgebra structure on $\cH$ and to avoid trivial cases. Similarly to the previous definition, we can equivalently define $G^N_{\omega}$, the group of truncated character on $\cH^{N}_{\omega}$ and we obtain the corresponding notion of rough path associated to a compatible weight.
\begin{defn}\label{def_inhom_rough_path}
Let $\gamma\in (0,1)$ and  $N_{\gamma}\geq 1$ defined as before. Let also  $\omega$ be  a compatible weight for $\cH$. A \emph{$\gamma$-weighted rough path over $\cH$ with respect to $\omega$} is a function $X\colon[0,T]^2\to (\cH^{N_{\gamma}}_{\omega})^*$ satisfying Chen's identity \eqref{eq:chen} at the level of $(\cH^{N_{\gamma}}_{\omega})^*$ and the following two properties 
\begin{itemize}
\item for all $x\in\cH^{N_{\gamma}}_{\omega}$, $y\in\cH^{N_{\gamma}}_{\omega}$ such that $xy\in \cH^{N_{\gamma}}_{\omega}$ one has 
\begin{equation}\label{eq:charachter_truncated_rough_in}\langle X_{st},xy\rangle=\langle X_{st},x\rangle\langle X_{st},y\rangle \,;
\end{equation}
for any $s,t\in [0,T]$;
\item for all  $u\in \cB^{N_{\gamma}}_{\omega}$
\begin{equation}\label{eq:genrpbound_in}
\sup_{s\neq t}\frac{|\langle X_{st}, u\rangle|}{\abs{t-s}^{\gga \omega(u) }}<+\infty\,.
\end{equation}
\end{itemize}
Denoting  by $\{e_i\}_{i=1,\ldots,d}$ the  basis of $\cH_{1,\omega}^{N_{\gamma}} $, if a path $x\colon[0,T]\to\bR^d$, $x= \{x^i\}_{i=1,\ldots,d}$ satisfies $x^i_t-x^i_s=\langle X_{st},e_i\rangle$, we say that $X$ is over $x$.
\end{defn}
\begin{rk}
This notion generalises the H\"older exponents of the functions $s,t\mapsto\langle X_{st}, u\rangle$ according to $\omega$ and it extends Definition \ref{dfn:genRP} (it is sufficient to consider $\omega=g$). Some explicit examples of $\gga$-weighted rough paths will be presented in the following sections. Similar families of rough paths with similar properties were presented in \cite{Tapia2018} with the denomination of \emph{anisotropic rough paths} and \cite{Gyurko2016} with \emph{$\pi$-rough paths}.
\end{rk}
Given the notion of a rough path , we introduce the associated space of paths that behave ``locally" like a rough path,  where we can define some analytic operations. This definition extends standard results contained in \cite{Gubinelli2004,Friz2020course,gub10,hairer2015geometric} in our algebraic context.
\begin{defn}\label{defn_contr_geo}
Let $X$ be a $\gga$-weighted rough path over $\cH$ with respect to some compatible weight $\omega$. For any $1\leq N \leq N_{\gamma}+1 $ a path $Y\colon[0,T]\to \cH^{N-1}_{\omega}$ is said a \emph{weighted controlled rough path} if for any  $u \in \cB^{N-1}_{\omega}$ the path $t\mapsto \langle u^*,Y_t\rangle$ is a $\gamma$-H\"older function and one has
\begin{equation}\label{contr_eq}
\sup_{s\neq t\in [0,T]^2}\frac{\abs{\langle u^*,Y_t\rangle- \langle X_{st}* u^*, Y_s\rangle}}{|t-s|^{(N_{\gga}-\omega(u))\gga}}<+\infty\,.
\end{equation}
In case $X$ is $\gamma$-rough path over $\cH$, we call controlled rough path every path $Y\colon[0,T]\to \cH^{N_{\gga}-1}$ satisfying the same property \eqref{contr_eq}, where we replace $\omega$ with the natural weight $|\cdot|$. Depending on the nature of $X$, we denote by  $\cD^{N\gamma}_X$ the space of weighted controlled rough paths/controlled rough paths. We say that a path $y\colon[0,T]\to\bR$ is controlled by $X$ if there exists a weighted  controlled rough path/controlled rough path $Y$ with $N>1$ which satisfies $y_t=\langle \1^*,Y_t\rangle$.
\end{defn}

\begin{rk}
Using the equivalent formulation of  $G^N_{\omega}$ in terms of a group,  we can  identify every $\gamma$-weighted rough path $X$ to the family of operators $\{\Gamma_{rv}\}_{r,v\in [0,T]^2}\colon \cH^{N_{\gamma}-1}_{\omega} \to  \cH^{N_{\gamma}-1}_{\omega}$ defined by 
\[\Gamma_{rv}:=  ((X_{rv})^{-1}\otimes id )\Delta\,,\]
where the inverse calculated with respect to the group $G^{N_{\gamma}}_{\omega}$. By means of Chen relations \eqref{eq:chen} we have the identity $(X_{rv})^{-1}= X_{vr}$ and  the equation \eqref{contr_eq} can be rewritten as
\begin{equation}\label{contr_eq_Gamma}
\sup_{s\neq t\in [0,T]^2}\frac{\abs{\langle u^*,Y_t - \gG_{ts} Y_s\rangle}}{|t-s|^{(N-\omega(u))\gga}}<+\infty\,.
\end{equation}
This way of reformulating Definition \ref{defn_contr_geo}, together with Definition \ref{def_inhom_rough_path}, allows us to rephrase these concepts via the theory of regularity structures again (see \cite{Hairer2014}). In particular, every  $\gga$-weighted rough paths $X$ is an example \emph{model} over the regularity structure $(\cH^{N_{\gamma}}_{\omega},G^{N_{\gamma}}_{\omega} )$ see \cite[Def. 2.17]{Hairer2014}. Furthermore, weighted controlled rough paths are a specific example of \emph{modelled distributions} \cite[Def. 3.1]{Hairer2014}. Using some standard argument of rough path theory (see e.g. \cite[Thm 4.3]{Manchon2018}) we can also uniquely extend $\gG$ to a family of linear maps $\{\Gamma_{rv}\}_{r,v\in [0,T]^2}\colon \cH\to \cH$ which we denote in the same way.
\end{rk}

In order to extend the notion of weighted controlled rough path above a vector-valued path $y\colon [0,T]\to \mathbb{R}^e$, we usually consider paths $Y$ taking values in $(\cH^{N-1}_{\omega})^e$ and we reinterpret  $\langle u^*,Y_t\rangle$ as a vector of $\bR^e$ which we denote by
\[ \langle u^*,Y_t\rangle=(\langle u^*,Y_t\rangle_1,\dotsc,\langle u^*,Y_t\rangle_e). \]
Then we require the bound \eqref{contr_eq} component-wise. We denote this space by $(\cD^{N\gamma}_X)^e$.

Weighted controlled rough paths are the main objects where we can extend the usual operation of composition and integration in this wider setting. In case of composition, for any fixed $Y\in (\cD^{N\gamma}_X)^e$ over a path $y\colon [0,T]\to\bR^e$ and a smooth function $ \gp\colon \bR^e\to \bR$, we want to define an element $Y^{\gp}\in \cD^{N\gamma}_X$ such that $\langle \1^*,  Y^{\gp}_t\rangle=\gp(y_t)$. To define it, we introduce the so-called \emph{lifting} of $\gp$, $\gP\colon (\cH^{N-1}_{\omega})^e\to \cH^{N-1}_{\omega}$ given for any $a\in (\cH^{N-1}_{\omega})^e$ by
\begin{subequations}
\begin{equation}\label{composition_ope}
\gP(a):= \sum_{n=0}^{m}\frac{1}{n!}\sum_{i_1,\cdots,i_n=1}^e\partial_{i_1}\cdots\partial_{i_n}\gp(\langle \1^*, a\rangle)[(a^+)_{i_1}\cdots (a^+)_{i_n}]_{<N}\,, 
\end{equation}
where $a^+= a- \langle \1^*, a\rangle$ is an element of $(\cH^{N-1}_{\omega})^e$, the product $\cdots$ is taken with respect to the intrinsic product $m$, $[\cdots]_{<N}$ is the projection on $\cH^{N-1}_{\omega}$ and we use the shorthand notation
\[\partial_{i_1}\cdots\partial_{i_n}\gp(x)= \frac{\partial^n}{\partial x_{i_1}\cdots\partial x_{i_n}}\gp(x)\,.\]
Since the product $m $ is commutative and the expression inside the second sum is symmetrical in the choice of $i_1$, $\cdots$ ,$i_n$, this definition is equivalent to the following
\begin{equation}\label{composition_ope2}
\gP(a)=\sum_{k}\frac{\partial^{k}\gp(\langle \1^*, a\rangle)}{k!}[(a^+)_{1}^{k_1}\cdots (a^+)_{e}^{k_e}]_{<N}\,,
\end{equation}
\end{subequations}
where the sum is done over all multi-indices $k\in \bN^e$, $k= (k_1, \cdots, k_e)$ and we adopt the notation
\[\partial^{k}\gp(x):=\partial_{1}^{k_1}\cdots \partial_{e}^{k_e}\gp (x)= \frac{\partial^{ k_{1}}}{\partial x_{1}^{k_1}}\cdots\frac{\partial^{ k_e}}{\partial x_{e}^{k_e}}\gp (x)\,, \quad k!=k_1!\cdots k_e!\, .\]
(See \ref{identity_sum} for the equivalence of the two notions). To formulate a rigorous result, we denote by $C^{k}_b(\bR^{e}, \bR)$ the set of functions $\gp\colon \bR^e\to \bR$ whose partial derivatives up to any multiindex of order $k$ are continuous and bounded.
\begin{prop}\label{composition}
Let $X$ be an  $\gga$-weighted rough path over $\cH$ with respect to some compatible weight. For any $Y\in (\cD^{N\gamma}_X)^e$ and  $ \gp\in C^{N}_b(\bR^{e}, \bR)$, $\gP(Y_t)$ belongs to $\cD^{N\gamma}_{X}$.
\end{prop}

\begin{proof}
The proof of this result was already given in \cite[Lem. 8.4]{gub10} in a specific case and it is also a direct consequence of \cite[Thm 4.16]{Hairer2014}. We repeat it here for sake of completeness. Following the identity \eqref{contr_eq_Gamma} the goal is to prove the estimate
\[\sup_{s\neq t\in [0,T]^2}\frac{\abs{\langle u^*,\gP(Y_t) - \gG_{ts} \gP(Y_s)\rangle}}{|t-s|^{(N-\omega(u))\gga}}<+\infty\,\]
for any $u \in \cB^{N-1}_{\omega}$.  Writing $Y_s= \widetilde{Y}_s+ Y_s^0\1$ for some fixed function $\widetilde{Y}\colon [0,T]\to( (\cH^{N-1}_{\omega}\cap \text{Ker}(\1^*))^e$, $\widetilde{Y}=(\widetilde{Y}^1,\cdots, \widetilde{Y}^e)$ and $Y^0\colon [0,T]\to \bR^e$, we use the multi-index notation \eqref{composition_ope2} to have the identity
\[\gP(Y_s)= \sum_{\abs{k}<  N}\frac{\partial^{k}\gp(Y^0_s)}{k!}((\widetilde{Y}_s)^{k})_{<N}\,,\]
where $\abs{k}=\sum_{i=1}^ek_i$. Applying the linear maps $\gG_{ts} $ to $\gP(Y_s)$, we use the intrinsic properties of the operators $\gG$ to deduce the existence of  a function $R\colon [0,T]^2\to(\cH^{N-1}_{\omega})^e$, $R=(R^1,\cdots, R^e) $ such that for any $i=1\,, \cdots \,, e$  and $u\in \cB^{N-1}_{\omega}$ one has
\begin{equation}\label{proof_composition1}
\gG_{ts}\gP(Y_s)= \sum_{\abs{k}<  N}\frac{\partial^{k}\gp(Y^0_s)}{k!}(\gG_{ts}\widetilde{Y}_s)^{k}+ R_{st}\,,\quad \sup_{s\neq t}\frac{\abs{\langle u^*,R^i_{st}\rangle}}{|t-s|^{(N-\omega(u)) \gamma}} <+\infty\,,
\end{equation}
By hypothesis on $Y$ and using the trivial identity $\gG_{ts}\1=\1$ there exists another function $\tilde{R}\colon [0,T]^2\to(\cH^{N-1}_{\omega})^e$, $\tilde{R}=(\tilde{R}^1,\cdots, \tilde{R}^e) $ such that one has similarly
\begin{equation}\label{proof_composition2}\gG_{ts}\widetilde{Y}_s= \gG_{ts}Y_s- Y^0_s\1=\widetilde{Y}_t + (Y^0_t- Y^0_s)\1 + \tilde{R}_{st}\,,
\end{equation}
where $\tilde{R}$ satisfies the same property of $R$. Combining the identities \eqref{proof_composition1} and \eqref{proof_composition2} there exists another function  $R'\colon [0,T]^2\to(\cH^{N-1}_{\omega})^e$ similar to $R$ and $\tilde{R}$ such that
\begin{equation}\label{proof_composition3}
\gG_{ts}\gP(Y_s)=  \sum_{\abs{k}<  N}\frac{\partial^{k}\gp(Y^0_s)}{k!}(\widetilde{Y}_t + (Y^0_t- Y^0_s)\1)^{ k}+R'_{st}\,.
\end{equation}
Performing a Taylor expansion between $Y^0_s$ and $Y^0_t$, we have for any multi-index $k$
\[
\partial^{k}\gp(Y^0_s)= \sum_{\abs{k+l}< N} \frac{\partial^{k+l}\gp(Y^0_t)}{l!}(Y^0_s-Y^0_t)^{ l} + O(\abs{t-s}^{(N-\abs{k})\gga})\,,
\]
as a consequence of the $\gga$-H\"older regularity of $Y$. Using the same bound we obtain also for any $u\in \cB^{N-1}_{\omega}$ that there exists a constant $C_{u}>0$ such that
\begin{equation}\label{proof_composition4}
|\langle u^*, (\widetilde{Y}_t + (Y^0_t- Y_s^0)\1)^{k}\rangle|\leq C_{u}\abs{t-s}^{(\abs{k}- \omega(u))\gga}\,.
\end{equation}
Plugging the bound \eqref{proof_composition4} and the Taylor formula into \eqref{proof_composition3} it is easy to show for any $u\in \cB^{N-1}_{\omega}$ that there exists a constant $C'_{u}>0$ such that
\[|\langle u^* , \gG_{ts}\gP(Y_s)- \sum_{\abs{k+l}< N }C_u\frac{\partial^{k+l}\gp(Y^0_t)}{k!\,l!}(Y^0_s-Y^0_t)^{ l}(\widetilde{Y}_t + (Y^0_t- Y^0_s)\1)^{ k}\rangle\leq C'_u \abs{t-s}^{(N-\omega(u))\gga}\,.\]
From this bound, it is sufficient to apply for any multi index $m$ the multinomial identity
\[\sum_{k+l=m} \frac{1}{k!\,l!}(Y^0_s-Y^0_t)^{ l}(\widetilde{Y}_t + (Y^0_t- Y^0_s)\1)^{ k}= \frac{\partial^m\gp(Y^0_t)}{m!}(\widetilde{Y}_t )^{ m}\,,\]
thereby obtaining the thesis.
\end{proof}
We pass to the notion of rough integration in this abstract context. To define it, we introduce firstly an algebraic notion of integration map compatible with our structure.
\begin{defn}\label{defn_compatible weight}
Let   $\omega$ be a compatible weight for $\cH$ and $n\geq 1$ an integer. A linear map $\cI\colon \cH \to \cH$  is said an integration map of order $n$ for $\omega$ if it satisfies the following properties:
\begin{itemize}
\item for every   $u \in \cB$ we have $\cI (u)\in \cB$ and $\omega(\cI (u))=  \omega(u)+n$, 
\item  for every $\tau\in \cH$ one has the identity
\begin{equation}\label{algebraic_identity}
\gD\cI (\tau)= \cI(\tau) \otimes \1+ (id \otimes \cI)\Delta \tau\,. 
\end{equation}
\end{itemize}
\end{defn}
\begin{rk}
The notion of integration map developed here is a specific case of abstract integration map of order $n$, as developed in the much more general framework of regularity structure, see \cite{Hairer2014}. Accordingly to \cite{Connes1998}, the property \eqref{algebraic_identity} can be rephrased in terms of the Hochschild cohomology of $\cH$. 
\end{rk}
Given this definition, we can state the general operation of rough integration at the level of weighted controlled rough paths
\begin{prop}\label{rough_integration}  
Let $X $ be an  $\gga$-weighted rough path over $\cH$ with respect to some compatible weight $\omega$ and let $\cI$ be an abstract integration map of order $n\leq N_{\gamma}$. Then for every $Y\in \cD^{(N_{\gamma}-n+1)\gamma}_{X}$ there exists a unique ${n\gamma}$-H\"older path  $I \colon [0,T]\to \bR$ given by 
\begin{equation}\label{eqn:rintdef}
t\mapsto I_t :=\lim_{|\pi|\to 0}\sum_{[s,u]\in\pi}\sum_{v\in\cB^{N-n}_{\omega}}\langle v^*,Y_s\rangle \langle X_{su}, \cI(v)\rangle\,,
\end{equation}
where $\pi$ is a sequence of partitions of $[0,t]$ whose mesh $|\pi|$ converges to $0$. We call it \emph{the rough integral of $Y$ with respect to $\cI$}. Moreover one has the estimate
\begin{equation}\label{equ:rough_int_est}
  I_t- I_s=\sum_{v\in\cB^{N-n}_{\omega}}\langle u^*,Y_s\rangle \langle X_{st}, \cI(u)\rangle + O(\abs{t-s}^{N_\gamma+1})
\end{equation}
for any $s<t$. Introducing the function $\cK (Y)_t:= I_t \1+ \cI(Y_t)$
one has  $ \cK (Y)\in \cD^{(N_{\gamma}+1)\gamma}_{X}$.
\end{prop}
This result extends \cite[Thm 8.5]{gub10} into a wider algebraic context and its proof is a standard application of the well-known \emph{Sewing Lemma} (see \cite{Gubinelli2004}), we will recall it as formulated in \cite[Prop. 2.1]{Manchon2018}, referring the reader to the same citation for the proof.
\begin{lem}[Sewing Lemma]\label{sewing_lemma}
Let $\cZ\colon[0,T]^2\to \bR$ be a continuous function and $\varepsilon>0$. Suppose that there exists a collection of real numbers $a_i,b_i\geq 0$ for $i=1,\ldots,n$ and $\beta\in (0,1)$ such that $a_i+b_i=1+\varepsilon$ and
\begin{equation}\label{hyp_sewing}
	\abs{\cZ_{st}-\cZ_{su} - \cZ_{ut}}\leq \sum_{i=1}^nC_i\vert t-u\vert^{a_i}\vert u-s\vert^{b_i}\,,\quad \abs{\cZ_{st}}\leq C_0 \abs{t-s}^{\beta}
\end{equation}
for some positive constants $\{C_i\}_{i=1,\cdots n}\geq 0$ and $C_0\geq 0$ and uniformly on $s,t,u\in[0,T]$ with $s\leq u\leq t$ or $t\leq u\leq s$. Then there exists a unique $\beta$-H\"older function $Z\colon [0,T]\to \bR$ such that $Z_0=0$ given by 
\begin{equation}\label{definition_integral}
t\mapsto Z_t= \lim_{\abs{\pi}\to 0}\sum_{[s,u]\in\pi}\cZ_{su}\,,
\end{equation}
where $\pi$ is a generic sequence of partitions of $[0, t]$ whose mesh $\abs{\pi}$ converges to $0$. Moreover there exists a constant $C>0$ such that one has the estimate
\[
\abs{Z_t-Z_s- \cZ_{st}}\leq C\abs{t-s}^{1+\gep}\,.
\]
\end{lem}


\begin{proof}[Proof of Proposition \ref{rough_integration}]
The result is obtained by checking that the function
\[\cZ_{st}:=\sum_{v\in\cB^{N-n}_{\omega}}\langle v^*,Y_s\rangle \langle X_{st}, \cI(v)\rangle\]
satisfies the hypotheses in the Sewing lemma. For any fixed $s\leq u\leq t$ one has by definition.
\[
\begin{split}
\cZ_{st}- \cZ_{su}- \cZ_{ut}&=\sum_{v\in\cB^{N-n}_{\omega}}\langle v^*,Y_s\rangle \bigg( \langle X_{st}, \cI(v)\rangle-  \langle X_{su}, \cI(v)\rangle\bigg)  -\langle v^*, Y_{u} \rangle \langle X_{ut}, \cI(v)\rangle\,.
\end{split}
\]
Let us fix a word $v\in\cB^{N-n}_{\omega}$. Using  Chen's property \eqref{eq:chen} and the definition of convolution product $*$ we have
\[\begin{split}
\langle X_{st}, \cI(v)\rangle- \langle X_{su}, \cI(v) \rangle&=\langle X_{su}* X_{ut}, \cI(v)\rangle- \langle X_{su},\cI(v)\rangle\\&=\langle X_{su}\otimes X_{ut}, \Delta (\cI(v))\rangle- \langle X_{su},\cI(v)\rangle\\&=\langle X_{su} \otimes X_{ut}, \cI(v)\otimes \1+(id \otimes \cI)\Delta v) \rangle- \langle X_{su},\cI(v)\rangle\\&=\langle X_{su},v^{(1)}\rangle \langle X_{ut}, \cI(v^{(2)})\rangle\,,
\end{split}\]
where we used the notation $\Delta v=  \sum v^{(1)}\otimes v^{(2)}$. Therefore we can write
\begin{equation}\label{new_equation}
\cZ_{st}- \cZ_{su}- \cZ_{ut}= \sum_{v\in\cB^{N-n}_{\omega}}\langle v^*, Y_s\rangle\langle X_{su},v^{(1)}\rangle \langle X_{ut}, \cI(v^{(2)})\rangle-\langle v^*,Y_{u} \rangle \langle X_{ut}, \cI(v)\rangle\,.
\end{equation}
In order to understand the right hand side of \eqref{new_equation} we split the sum in two terms. Writing on the same dual basis the first term becomes
\[
\begin{split}
\sum_{v\in\cB^{N-n}_{\omega}}\langle v^*,Y_s\rangle  \langle X_{su},v^{(1)}\rangle \langle X_{ut}, \cI(v^{(2)})\rangle&= \sum_{v,k\in\cB^{N-n}_{\omega}}\langle v^*,Y_s\rangle\langle X_{su},v^{(1)}\rangle\langle k^*, v^{(2)}\rangle\langle X_{ut}, \cI(k)\rangle\\&=\sum_{v,k\in\cB^{N-n}_{\omega}}\langle v^*,Y_s\rangle\langle X_{su}\otimes k^*, \Delta v\rangle\langle X_{ut}, \cI(k)\rangle\\&=\sum_{k\in\cB^{N-n}_{\omega}}\langle X_{su}* k^*, Y_s\rangle\langle X_{ut}, \cI(k)\rangle\,.
\end{split}
\]
Since the index in the summation is mute, we can finally apply the bound and the H\"older property of a rough path in to obtain the existence of a sequence  of positive constants $\{C_{v}\}_{v\in\cB^{N-n}_{\omega}}$ such that
\[\begin{split}
\abs{\cZ_{st}- \cZ_{su}- \cZ_{ut}}&\leq \sum_{v\in\cB^{N-n}_{\omega}}\abs{\big(\langle X_{st}* v^*, Y_s\rangle-\langle v^*, Y_u\rangle\big)\langle X_{ut}, \cI(v)\rangle}\\&\leq \sum_{v\in\cB^{N-n}_{\omega}} C_{w}\abs{u-s}^{(N_{\gamma}+1-n- \omega(v) )\gamma} \abs{t-u}^{(\omega(v)+n)\gga}\,.
\end{split}\]
By construction one has $(N_{\gamma}+  1) \gamma >1$ and  we can apply the sewing lemma to obtain the function $I_t$. The second part of the statement comes trivially from Definition \ref{defn_contr_geo}.
\end{proof}
\begin{rk}\label{rk_independence_product}
Looking at the proof of this result,  the definition of the rough integral depends only on the coproduct structure and the weight $\omega$. Thus the construction holds independently of two product structure we considered.
\end{rk}
\subsection{Examples of rough paths}
Depending on the choice of the underlying Hopf algebra $\cH$ and integration maps $\cI$, we recover some well-known definitions in the literature of rough paths.

\subsubsection*{Geometric rough paths} 
For any real vector space $V$ with finite dimension, we consider its tensor algebra $T(V)$ given by
\[T(V)= \bigoplus_{k=0}^{+\infty}V^{\otimes k}\,,\]
where by convention we set $V^{\otimes 0}= \bR$. To simplify the notation of pure tensors, we fix a basis of $V$ and we identify it with $\bR^{A}$, the free vector space generated from the finite set $A=\{1\,, \cdots\,, \dim (V)\}$. We will equivalently use the word \emph{alphabet} to denote a finite set. Extending this identification at the level of the tensor algebra, $T(V)$ is isomorphic to $T(A)$, the free vector space generated from the set of \emph{words} built from $A$ union the empty word $\1$, we denote it with $W(A)$.

Two operations can be naturally defined on the tensor algebra: the shuffle product $\shuffle$, defined recursively from the identities $\1 \,\shuffle\, v= v\, \shuffle\,\1 =v$ for any $v\in T(A) $ and for any couple of words $v,w\in W(A)$ and letters $a,b\in A$
\begin{equation}\label{recursive_shuffle}
va \shuffle wb= (v \shuffle wb)a + (va \shuffle w)b \,,
\end{equation}
where $va$ and $wb$ is the juxtaposition of the letters $a,b$ with the  words $v,w$. On the other hand, we introduce the  deconcatenation coproduct $\nabla$, defined by the identity $\nabla \1= \1 \otimes \1\\ $
and the relation
\begin{equation}\label{def_deconcatenation}
\nabla w=\1\otimes w +w\otimes \1+ \sum_{i=1}^{m-1} a_1\cdots a_i \otimes a_{i+1}\cdots a_m \,,
\end{equation}
for any non-empty word $w=a_1\cdots a_m\in W(A)$. It is a classical result in the algebraic literature (see e.g. \cite{reutenauer1993free}) that the triple $(T(A), \shuffle, \nabla)$ is a graded commutative bialgebra with unity $\1$ and the natural grading is given by the couple $(|\cdot|, W(A))$ where $|\cdot|$ is the word \emph{length}. Following the definition  of $\nabla$,  the maps $(\cI_i)_{i\in A}\colon T(A)\to T(A)$ defined by $\cI_i(w)=wi$ on $W(A)$ are integrations map of order $1$ for $|\cdot|$ for any $i \in A$.

Applying Definition \ref{dfn:genRP} and Proposition \ref{rough_integration}, we obtain the notion geometric rough paths and rough integration as defined in \cite{hairer2015geometric, bellingeri2020transport}.
\begin{defn}\label{defn_geom_rough}
Let $\gamma\in (0,1)$ and $A$ being an alphabet. We call every  $\gamma$-rough path over $T(A)$ a $\gamma$-\emph{geometric rough path}. For any $\gamma$-geometric rough path $X$, $i \in A$ and any $Y\in \cD^{N_{\gamma}\gamma}_{X}$ we call the rough integral of $Y$ with respect to $\cI_i$ the \emph{geometric rough integral} of $Y$ with respect to $X^i$, which is given by the identity
\begin{equation}\label{defn_rough_geom_int}
\int_{0}^t Y_r d^gX^i_r:= \lim_{|\pi|\to 0}\sum_{[s,u]\in\pi}\,\sum_{\substack{w\in W(A)\\|w|\leq N_{\gamma}-1}}\langle w^*,Y_s\rangle \langle X_{su}, wi \rangle\,.
\end{equation}
In case $\omega$ is  a compatible weight for $T(A)$, we call every  $\gga$-weighted rough path over $T(A)$ with respect to some compatible weight a \emph{$\gamma$-weighted geometric rough path}.
\end{defn}
\subsubsection*{Branched rough paths}
For any given  alphabet $A$, an $A$-labelled rooted tree is a non empty rooted combinatorial tree $\tau$ whose nodes are labelled by the elements of $A$. The set of $A$-labelled trees is denoted by $\cT(A)$. A finite disjoint union of $A$-labelled rooted trees or the empty graph $\1$ (just the same notation of the empty word) is called an $A$-labelled \emph{forest} and the set of all $A$-labelled forest is denoted by  $\cF(A)$. For any finite family of  trees $\t_1\,,\cdots \,,\t_n$ the forest obtained by the disjoint union of the previous trees is denoted by $f=\t_1\cdots \t_n$ independently of the order of the trees. We can graphically represent elements of $A$ by simply  putting the root at the bottom and decorating each node with the corresponding label. 
For instance if $A=\{1\,,2\}$ one has
\[   \begin{tikz}[dtree]
    \node[dtree black node, label={ right: \tiny 1}]{};
  \end{tikz}, \begin{tikz}[dtree]
  \node[dtree black node,  label={ right: \tiny 1}] {}
  child { node[dtree black node,label={ right: \tiny 2}] {} }  ;
\end{tikz}, \begin{tikzpicture}[dtree]
    \node[dtree black node,label={ right: \tiny 2}] {} 
    child { node[dtree black node,label={left: \tiny 1}]{} }
    child { node[dtree black node,label={ right: \tiny 1}]{} }
    ;
  \end{tikzpicture}\in \cT(A)\,;\qquad  \begin{tikzpicture}[dtree]
    \node[dtree black node,label={ right: \tiny 1}] {} 
    child { node[dtree black node,label={left: \tiny 2}]{} }
    child { node[dtree black node,label={ right: \tiny 2}]{} }
    ;
  \end{tikzpicture}\begin{tikz}[dtree]
  \node[dtree black node,  label={ right: \tiny 1}] {}
  child { node[dtree black node,label={ right: \tiny 2}] {} }  ;
\end{tikz} \in \cF(A)\,.\] 
Starting from the empty graph $\1$ we can recursively define $\cT(A)$ by means of the \emph{grafting} maps $(B^+_i)_{i\in A} \colon \cF(A)\to \cT(A)$. That is for any $i\in A$  we set $ B^+_i(\1):=\bullet_i$ and for any forest $\t_1\cdots \t_n$ the tree  $B^+_i(\t_1\cdots \t_m)$ is graphically given by
\[B^+_i(\t_1\cdots \t_m)=  \begin{tikzpicture}[dtree]
    \node[dtree black node,  label={ right: $ \scriptstyle i$}] {}
    child { node[dtree black node, label={ above: \tiny $\t_{1}$}]{} }
    child { node[dtree black node]{} }
    child { node[dtree black node, label={[label distance=0.03cm]90: $\cdots$ }]{} }
    child { node[dtree black node]{} }
    child { node[dtree black node, label={ above: \tiny $\t_{m}$}]{} }
    ;
  \end{tikzpicture}\;.\]
(since we consider combinatorial graphs, all the graphical representations of $B^+_i(\t_1\cdots \t_m)$ are identified to a single tree). We denote by $\cH(A)$  the free vector space generated from $\cF(A)$. This vector space has also a natural bialgebra structure.
The product coincides with the linear extension of the disjoint union of graphs and we denote it as a simple juxtaposition. On the other hand, the coproduct  $\gD$ can be uniquely defined by the base condition $ \gD\1= \1\otimes\1$ and  the recursive identities
\begin{equation}\label{defn_BCK_coproduct}
\gD f:=\gD\t_1\cdots \gD\t_n\,, \quad \gD(B^+_i(\gs)):=  B^+_i(\gs)\otimes\1 + (\text{id}\otimes B^+_i)\gD\gs\,,
\end{equation}
for any forest $f=\t_1\cdots \t_n$, $\gs\in \cF(A)$ and $i \in A$. It comes easily from the respective definitions that the resulting triple $(\cH(A), \cdot, \gD)$ is still a graded commutative bialgebra with unity $\1$ and the natural grading is given by the couple $(|\cdot|, \cF(A))$ where $|\cdot|$ is the forest \emph{cardinality}. The grafting maps $B^+_i$ are integration maps of order $1$ for $|\cdot|$.

The resulting Hopf algebra is known in the literature as the \emph{Butcher-Connes-Kreimer Hopf Algebra} (see \cite{butcher72,Connes1998}) and it has been used to introduce the class of branched rough path (see e.g. \cite{hairer2015geometric,gub10}). From Definition \ref{dfn:genRP} and Proposition \ref{rough_integration} we obtain the following definition.
\begin{defn}\label{defn_branched}
Let $\gamma\in (0,1)$ and $A$ being an alphabet. We call every  $\gamma$-rough path over $\cH(A)$ a $\gamma$-\emph{branched rough path}. For any $\gamma$-branched rough path $X$, $i \in A$ and any $Y\in \cD^{N_{\gamma}\gamma}_{X}$  we call the rough integral of $Y$ with respect to $B^+_i$ the \emph{branched rough integral} of $Y$ with respect to $X^i$, which is given by the identity
\begin{equation}\label{defn_rough_int_br}
\int_{0}^t Y_r d^bX^i_r= \lim_{|\pi|\to 0}\sum_{[s,u]\in\pi}\,\sum_{\substack{\tau\in \cF(A)\\|\tau|\leq  N_{\gamma}-1}}\langle \tau^*,Y_s\rangle \langle X_{su},B_i^+(\tau) \rangle\,.
\end{equation}
In case $\omega$ is  a compatible weight for $\cH(A)$, we call every  $\gga$-weighted rough path over $\cH(A)$ with respect to some compatible weight a \emph{$\gamma$-weighted branched rough path}.
\end{defn}

\subsubsection*{Random rough path}
All definitions in rough path theory are deterministic but the most interesting examples are obtained when we consider stochastic biprocesses $X\colon[0,T]^2\to (\cH^{N}_{\omega})^*$, depending on a complete probability space $(\gO, \cF, \bP)$.  We present  the generalisation of a standard criterion to prove that a biprocess  has a.s. the H\"older property \eqref{eq:genrpbound_in}. The whole result is based on  this deterministic inequality. For its proof we refer to 
\cite[Cor. 4]{Gubinelli2004}.
\begin{lem}\label{gub_Lemma}
Let $\alpha>0$ and $p\geq 1$. For any measurable function $Z\colon [0,T]^2\to \bR$ there exists a constant  $C(\alpha, p)>0$ depending on $\alpha$ and $p$ such that
\[||Z||_{\alpha}:=\sup_{s\neq t\in [0,T]^2}\frac{\abs{Z_{st}}}{|t-s|^{\alpha}}\leq C(\alpha, p)\left(\left[\int_0^T\int_0^T\frac{|Z_{st}|^{2p} }{|t-s|^{2\alpha p+4}}dsdt\right]^{\frac{1}{2p}} +|||\delta Z |||_{\alpha}\right)\,,\]
where the constant $|||\delta Z |||_{\alpha}$ is given by 
\[|||\delta Z |||_{\alpha} =\inf\left\{ \sum_{i=1}^n\sup_{s\leq u  \leq  t\in [0,T]^3} \frac{|Z^i_{st}- Z^i_{su}- Z^i_{ut}|}{|u-s|^{\rho_i}|t-u|^{\alpha- \rho_i}}\colon Z=  \sum_{i=1}^nZ_i\right\}\]
and the last infimum is taken over all sequences  such that $Z=  \sum_{i=1}^nZ_i$ and for all choices of the numbers $\rho_i\in (0,\alpha)$.
\end{lem}
Thanks to this Lemma, we can easily formulate a sufficient condition to check the definition of an $\gga$-weighted rough path in a random setting. 

\begin{thm}\label{Holder_criterion}
Let $\cH$ be locally finite graded and connected Hopf algebra endowed with  a compatible weight $\omega$ and let $X$ be stochastic biprocess $X\colon[0,T]^2\to (\cH^{N}_{\omega})^*$ satisfying a.s. Chen's identity \eqref{eq:chen} at the level of $(\cH^{N}_{\omega})^*$, the algebraic property \eqref{eq:charachter_truncated_rough_in} and $\langle X_{st} , \1\rangle=1$. Supposing that for every $v\in \cB^{N}_{\omega}$ there exists a constant $C_v>0$ depending on $v$ such that
\begin{equation}\label{eq_holder_bound}
\bE |\langle X_{st} , v\rangle|^{2p}\leq C_v|t-s|^{\frac{2p\omega(v)}{N}}\,.
\end{equation}
Then for any $1/(N+1)<\gamma<1/N$,  $X$ is a  $\gga$-weighted rough path a.s. .
\end{thm}
\begin{proof}
The only property to check  is to prove that  \eqref{eq:charachter_truncated_rough_in} holds a.s.  for any $v\in \cB^{N}_{\omega}$. We will prove the result by induction over the natural grading of $\cH$. We begin by considering $v\in \cH^{N}_1\cap \cB^{N}_{\omega}$. Since $\cH$ is graded one has the algebraic identity
\[\gD v= v\otimes \1+ \1 \otimes v\]
which becomes via the Chen identity $\langle X_{st} , v\rangle= \langle X_{su} , v\rangle+\langle X_{ut} , v\rangle$ a.s. Then we can apply Lemma \ref{gub_Lemma}, obtaining the a.s. inequality
\[\sup_{s\neq t\in [0,T]^2}\frac{\abs{\langle X_{st} , v\rangle}}{|t-s|^{\omega(v)\gamma}}\leq C(v, p)\left[\int_0^T\int_0^T\frac{|\langle X_{st} , v\rangle|^{2p} }{|t-s|^{2\omega(v)\gamma p+4}}dsdt\right]^{\frac{1}{2p}}\,\,,\]
for any $p \geq 1$ and some constant $C(v, p)>0$ depending on $v$ and $p$. Using Jensen inquality and hypothesis \eqref{eq_holder_bound}, there exists a constant $C'>0$ such that 
\begin{equation}\label{first_bound_proof}
\begin{split}
&\bE \left[\int_0^T\int_0^T\frac{|\langle X_{st} , v\rangle|^{2p} }{|t-s|^{2\omega(v)\gamma p+4}}dsdt\right]^{1/(2p)}\leq C'\left[\int_0^T\int_0^T |t-s|^{2(p\omega(v)(\frac{1}{N}-\gamma)- 2)}dsdt\right]^{\frac{1}{2p}}\,.
\end{split}
\end{equation}
The hypothesis on $\gamma$ allow to pick $p$ sufficiently big in the right-hand side of \eqref{first_bound_proof} such that the expectation is finite, thereby obtaining  the basis of induction
\[\sup_{s\neq t\in [0,T]^2}\frac{\abs{\langle X_{st} , v\rangle}}{|t-s|^{\omega(v)\gamma}}<\infty \quad \text{a.s.}\]
In case of a generic $w \in \cH^{N}_n\cap \cB^{N}_{\omega}$ $n\leq N$, we use the Sweedler notation for the reduced coproduct $\Delta' w= w^{(1)}\otimes w^{(2)}$. The grading of $\cH$ implies that every element $w^{(1)}$ or $w^{(2)}$ belong to $\cH^{N}_k$ for $k<n$ and the induction hypothesis tell us that the values $|| \langle X , w^{(1)} \rangle||_{\omega(w^{(1)})\gamma} $ and $||\langle X , w^{(2)} \rangle||_{\omega(w^{(2)})\gamma} $ are a.s. finite. Since Chen identity holds, one has the a.s. equality
\begin{equation}\label{eq_concrete_chen}
\langle X_{st} , w\rangle- \langle X_{su} , w\rangle-\langle X_{ut} , w\rangle= \langle X_{su} \otimes  X_{ut} , \gD'w\rangle= \langle X_{su} ,w^{(1)}\rangle\langle X_{ut} , w^{(2)}\rangle\,.
\end{equation}
Apply again Lemma \ref{gub_Lemma}, for any $p \geq 1$ there exists a constant $C(w, p)>0$ depending on $w$ and $p$ such that 
\[\sup_{s\neq t\in [0,T]^2}\frac{\abs{\langle X_{st} , w\rangle}}{|t-s|^{\omega(w)\gamma}}\leq C(w, p)\left(\left[\int_0^T\int_0^T\frac{|\langle X_{st} , w\rangle|^{2p} }{|t-s|^{2\omega(w)\gamma p+4}}dsdt\right]^{\frac{1}{2p}} +|||\delta \langle X , w\rangle |||_{\omega(v)\gamma}\right)\,.\]
Using the same reasoning as in \eqref{first_bound_proof} we can choose $p$ sufficiently by such that the integral in the above sum is a.s. finite. Moreover we deduce from identity \eqref{eq_concrete_chen} the following estimate
\begin{equation}
|||\delta \langle X , w\rangle |||_{\omega(w)\gamma} \leq || \langle X , w^{(1)} \rangle||_{\omega(w^{(1)})\gamma} ||\langle X ,w^{(2)} \rangle||_{\omega(w^{(2)})\gamma} <\infty \quad a.s.
\end{equation}
Since both sides are a.s. finite, the result is proven.
\end{proof}
\begin{rk}\label{Holder_criterion_infinite}
In case one considers a stochastic biprocess $X$ with values over $(\cH)^*$, satisfying a.s. the properties \eqref{eq:chen} and \eqref{eq:charachter_truncated_rough_in} over $\cH$ and the bounds \eqref{eq_holder_bound} for some integer $N>1$ and every $v\in \cB$, the same  proof allows to show that $X$ satisfies  also \eqref{eq:genrpbound_in} for any $1/(N+1)<\gamma<N$. In other words $X$ is a.s. a $\gamma$-regular $\cH$ rough path, using the terminology of  \cite{Manchon2018}. The key to understand this extension is due essentially to the hypothesis that $\cB$ is countable and the usual properties related to measurable sets of full probability.
\end{rk}
\section{Quasi-geometric rough paths}
We present here the definitions and some properties related to the class of quasi-geometric rough paths. These constructions can be applied immediately to rewrite some standard examples into this framework.
\subsection{Quasi-shuffle algebras}\label{Quasi_shuffle_algebra}
To introduce quasi-geometric rough paths, we recall the main properties of real quasi-shuffle algebras, as described in \cite{hoffman17}. The original formulation of the theory starts with a countable infinite set  $A$ and its  corresponding tensor algebra $T(A)$ built from that. These hypothesis are more general than our purposes  and in what follows we will always start from an alphabet $A$. To set up this product we also have to endow $A$ with a commutative and associative product $[\,,]$ on $\bR^{A}$. We refer to it as \emph{commutative bracket}. Thanks to the properties of a commutative bracket $[\,, ]$, for any  word $w\in W(A)$, $w= a_1\cdots a_n$ we adopt the notation $[a_1\cdots a_n]$ to denote the quantity $[a_1 [\cdots [a_{n-1}, a_{n}]]\cdots]$ independently on the order of the letters and the parenthesis. This additional structure combines into the definition of the quasi-shuffle product.
\begin{defn}
Let $A$ be an alphabet and $[\,, ]$ a commutative bracket. We define the \emph{quasi-shuffle product} $\qshuffle$   as the unique bilinear map in $T(A)$ satisfying relation $\1 \,\qshuffle\, v= v\, \qshuffle\,\1 =v$ for any $v\in T(A) $ and the recursive identity
\begin{equation}\label{quasi_shuffle_recursive}
va \qshuffle wb= (v \qshuffle wb)a + (va \qshuffle w)b+  (v\qshuffle w)[ab] \,,
\end{equation}
for any couple of words $v,w\in W(A)$ and letters $a,b\in A$.
\end{defn}
The first result we mention tells us that this operation has the good properties to be included in our algebraic context. For its proof see \cite[Thm. 2.1, 3.1]{hoffman2000}. 
\begin{thm}
For any choice of commutative bracket $[\,,]$ the couple $(T(A),\qshuffle)$ is a well-defined commutative algebra and  the triple $(T(A), \qshuffle, \nabla)$ is a commutative bialgebra.
\end{thm}
We will also adopt the shorthand notation $\hat{T}(A)$ to denote the tensor algebra $T(A)$ endowed with the quasi-shuffle product. The second property we recall a combinatorial identity of the quasi shuffle product. For any word $w=a_1\cdots a_l$  and a surjective map $f\colon \{1\,, \cdots \,, l\} \to \{1\,, \cdots \,, p\} $, $p\leq l$ we define the contracted word as
\[[a_1\cdots a_l]_f:= \bigg[\prod_{j\in f^{-1}(1)}a_j\bigg]\cdots \bigg[\prod_{j\in f^{-1}(p)}a_j\bigg]\,.\]
Thank to this notation we can use surjections to express $\qshuffle$. See \cite[Pag. 7]{Ebrahimi-Fard2015} as reference.
\begin{prop}\label{prop_alt_quasi_shuffle}
For any couple of non-empty words $w=a_1\cdots a_n$ and $v= b_1\cdots b_m$ one has the identity
\begin{equation}\label{defn_alt_quasi_shuffle}
w\qshuffle v= \sum_{f\in\bS_{n,m}} [a_1\cdots a_nb_1\cdots b_m]_{f}\,,
\end{equation}
where $\bS_{n,m} $ is the set of all surjections $f\colon \{1\,, \cdots \,, n+m\}\to \{1\,, \cdots \,, k\}$  satisfying $f(1)<\cdots < f(n), f(n + 1) <\cdots < f(n + m)$ for some integer $k$ such that $\max(m,n)\leq k\leq n+m$. 
\end{prop}
As before, we can use the word length to grade the tensor algebra but this graduation is not always compatible with $\qshuffle$. In case $(\bR^A, [\,,])$ can be graduated along strictly positive integers, we can easily define a compatible grading.
\begin{prop}\label{prop_natural_grading}
Let $A$ be an alphabet and $[\,,]$ a commutative bracket. We suppose there exists a family of vector spaces $(V_i)_{i>0}$ satisfying for any $i,j>0$
\begin{equation}\label{decomposition_R_A}
\bR^A= \bigoplus_{i>0} V_i\,,\quad [\,, ]\colon V_i\otimes V_j \to V_{i+j}\,.
\end{equation} 
Then there exists a weight $\omega$ for $T(A)$ such that $(T(A), \qshuffle, \nabla)$ is a commutative graded bialgebra. Moreover, $\omega$ is also compatible with $T(A)$ endowed with the word length.
\end{prop}
\begin{proof}
It follows from hypothesis \eqref{decomposition_R_A} and the finiteness of $A$ that there exists an integer $M\geq 1$  such that $\bR^A$ is isomorphic to $\bigoplus_{0<i\leq M} V_i$, each space is non zero, $ V_j=0$ for any $j>M$ and we still have the second property in  \eqref{decomposition_R_A}. Given such decomposition, we simply define the function $\omega\colon W(A)\to \bN $ as
\[ \omega(b)= i\,  \quad \text{if $b \in V_i\cap A$}\,, \qquad \omega(a_1\cdots a_l)= \sum_{j=1}^l\omega(a_j)\,, \qquad \omega(\1)=0\,. \]
The resulting function $\omega$ is trivially an weight for $T(A)$ which is also compatible with it when we consider the word length $|\cdot|$. The result will follow once we show
\[\qshuffle\colon \cH_n\otimes \cH_m\to \cH_{n+m} \]
for any couple of integers $m,n\geq 0$, where $\cH_n= \langle v\in W(A)\colon \omega(v)=n\rangle$. Thanks to Proposition\eqref{prop_alt_quasi_shuffle}, for any two non-empty words $v,u$ such that $\omega(v)=n$ and $ \omega(u)=m$ the product  $v\qshuffle u $  is a linear combination of terms $[vu]_f$ where  $\max(|v|,|u|)\leq k\leq |v|+|u|$ and $ f\in \bS_{|v|,|u|}$. For any choice of $f$ and $k$ such that $[vu]_f\neq 0$ one has by definition of $\omega$
\begin{equation}\label{proof_weight_1}
\omega( [vu]_f)= \sum_{l=1}^{k}\omega \bigg( \bigg[ \prod_{j\in f^{-1}(l)}(vu)_j\bigg]\bigg)\,,
\end{equation}
where $(vu)_j$ is the $j$-th letter in the word $vw$ and they satisfy for any $l$
\[\sum_{j\in f^{-1}(l)}\omega ((vu)_j)\leq M\,.\]
Using again property \eqref{decomposition_R_A}, the right-side of \eqref{proof_weight_1}  becomes 
\[\sum_{l=1}^{k}\sum_{j\in f^{-1}(l)}\omega ((vu)_j)= \sum_{l=1}^{|v|+|u|}\omega ((vu)_j)= m+n\,, \]
because the function $f$ is surjective.
\end{proof}
\begin{rk}
Property \eqref{decomposition_R_A} can be easily verified on all the examples we will show and it was included in the original definition of quasi-shuffle algebra \cite{hoffman2000}. Hence we will suppose that \eqref{decomposition_R_A} is always verified a priori and we call the resulting weight the intrinsic weight of $\hat{T}(A)$. We remark also that this weight is also compatible with the standard shuffle product $\shuffle$.
\end{rk}

Proceeding as before, one has that $(T(A), \qshuffle, \nabla)$ is indeed Hopf algebra.  Shuffle and quasi shuffle structures are intimately related. Indeed if $[\,,]$ is the function constantly equal to $0$, the quasi-shuffle product becomes trivially the shuffle product. The last result we recall  is the existence of an explicit isomorphism between $T(A)$ and $\qsT(A)$. To denote it we need some  combinatorial notations. For any integer $n>0$ we say that the multi-index $I=(i_1,\cdots, i_m)$ with all strictly positive components is a composition of $n$ if $ i_1+ \cdots+ i_m= n$. The set of all composition of $n$ is denoted by $C(n)$. Using the same multi-index notation, for any composition $I=(i_1,\cdots, i_m)$ we set
\[I!= i_1!\cdots i_m!\,,\quad I= i_1\cdots i_m\,, \quad  |I|= m\,.\]
Moreover for any word $w\in T(A)$, $w=a_1\cdots a_n$ and any $I\in C(n)$, $I=(i_1,\cdots, i_m)$ we define the \emph{contracted word} $[w]_I\in \bR^A$ as
\[ [w]_I:= [w_1\cdots w_{i_1}][w_{i_1+1}\cdots w_{i_2}]\cdots[w_{i_{m}+1} \cdots w_n]\,.\]
We can now state one of the most important results in the study of quasi-shuffle algebra. For its proof see \cite[Thm. 2.5]{hoffman2000}.

\begin{thm}\label{hoffmann_iso}
Let $A$ be an alphabet and $[\,,]$ a commutative bracket. We define the maps $\exp \,, \log\colon  T(A) \to  T(A)$ on any word $w \in W(A)$ as
\begin{equation}\label{defn_exp}
\exp(w):=\sum_{I\in C(\abs{w})}\frac{1}{I!}[w]_I\,, \quad\log(w):=\sum_{I\in C(\abs{w})}\frac{(-1)^{\abs{w}-|I|}}{I}[w]_I\,,
\end{equation}
The map $\exp$  is the unique Hopf algebra graded isomorphism between $T(A)$ endowed with the word length and $\qsT(A)$ endowed with the natural grading $\omega$. That is one has
\begin{equation}\label{prop_exp}
\begin{gathered}
\exp(w\shuffle v)= \exp(w)\qshuffle \exp(v)\,, \quad \nabla\exp(w)= (\exp \otimes \exp )\nabla w\,,\\
\exp\colon \langle v\in W(A)\colon |v|=n\rangle\to \langle v\in W(A)\colon \omega(v)=n\rangle\,,
\end{gathered}
\end{equation}
for any $v,w\in T(A)$. Moreover the inverse of $\exp$ is given by $\log$. We call these maps  \emph{exponential} and \emph{logarithm} of words.
\end{thm}
Using the structure of the set of $C(n)$, we can easily both applications $\exp$ and $\log$. For instance, for any triple of letters $a,b,c\in A$ the definitions in \eqref{defn_exp} become
\begin{equation}\label{explicit_exp_log}
\begin{gathered}
\exp(ab)= ab+ \frac{1}{2}[ab]\,,\quad \exp(abc)= abc+ \frac{1}{2!}[ab]c+ \frac{1}{2!}a[bc]+ \frac{1}{3!}[abc]\,;\\
\log(ab)= ab- \frac{1}{2}[ab]\, ,\quad \log(abc)=abc- \frac{1}{2}[ab]c- \frac{1}{2}a[bc]+ \frac{1}{3}[abc]\,.
\end{gathered}
\end{equation}
\subsection{Main properties}
We now have all the notions to define  quasi-geometric rough paths. In order to be coherent with the definition of geometric rough path,  we decide to grade $\hat{T}(A)$   following the word length $|\cdot|$ and we consider the natural grading $\omega$ given by Proposition \ref{prop_natural_grading} as a compatible weight for $T(A)$ (we recall that Definition \ref{defn_compatible weight} depends only on the coproduct $\nabla$). Thus our main definition becomes a simple specification of Definition \ref{def_inhom_rough_path}.
\begin{defn}\label{quasi_geom_rough}
Let $\gamma\in (0,1)$, $A$ an alphabet and $[\,,]$ a commutative bracket. We call every  $\gga$-weighted rough path over $\hat{T}(A)$ with respect to its intrinsic weight a \emph{$\gamma$-quasi-geometric rough path}.
\end{defn}
Looking at the rough integration of weighted controlled rough paths with respect to quasi-geometric rough path, we do not need to define any new notion of rough integration  because the coalgebra structure does not change (see \ref{rk_independence_product}).  In case when $\omega$ is different from the word length, the maps $(\cI_i)_{i\in A}\colon T(A)\to T(A)$ defined above are all integrations map of order $\omega(i)$ for $|\cdot|$ for any $i \in A$. Then the rough integral is defined as the limit
\begin{equation}
 \int_{0}^t Y_r d^gX^i_r= \lim_{|\pi|\to 0}\sum_{[s,u]\in\pi}\,\sum_{\substack{v\in W(A)\\ \omega(v)\leq N_{\gamma}- \omega(i)}}\langle v^*,Y_s\rangle \langle X_{su}, vi \rangle\,.
\end{equation}
We now use the isomorphism given by Theorem \ref{hoffmann_iso} to construct a explicit bijection between $\gga$-quasi-geometric rough paths and  $\gga$-weighted geometric rough path for any $\gga\in (0,1)$. In what follows, we denote by  $\exp^*\colon (\hat{T}(A))^*\to  (T(A))^*$ and $\log^*\colon (T(A))^*\to (\hat{T}(A))^*$ the dual maps of $\exp$ and $\log$. 

\begin{thm}\label{fund_thm_quasi}
Let $\gamma\in (0,1)$, $A$ an alphabet, $[\,,]$ a commutative bracket and  $\omega$ the intrinsic weight of $\hat{T}(A)$. Then the functions $\exp^*$ and $\log^*$ are a bijection between $\gga$-quasi-geometric rough paths and  $\gga$-weighted geometric rough paths.
\end{thm}
\begin{proof}
The proof follows trivially from the properties \eqref{prop_exp}. Indeed for any given $\gga$-geometric rough path $X\colon[0,T]^2\to (\hat{T}(A)^{N_{\gamma}}_{\omega})^*$, we can easily check that the function $\exp^*X$  satisfies the conditions defining a  $\gga$-weighted geometric rough path. Notably, from the first line in \eqref{prop_exp} we deduce that $\exp^*X$ satisfies the Chen property \eqref{eq:chen} and the multiplicative one  \eqref{eq:charachter_truncated_rough_in} with the shuffle product. From the fact that $\exp$ is a graded isomorphism, we obtain for any fixed word  $v\in W(A)$, $\exp(v)$ is a finite linear combination of words where $\omega$ does not change, thereby obtaining the property \eqref{eq:genrpbound_in}. Using the same reasoning,  for any $\gga$-weighted geometric rough path $Z\colon[0,T]^2\to (T(A)^{N_{\gamma}}_{\omega})^*$ the function  $\log^*Z$ is a well $\gga$-geometric rough path. Since $\log^*\exp^*=\exp^*\log^*=id^*$, we conclude.
\end{proof}
Thanks to this explicit bijection, we can easily adapt  two main properties of geometric rough paths in this new framework. Their names are given coherently with \cite{Lyons2007,lyons1998}.
\begin{prop}\label{ext_lyo_quasi}
Let $\gamma\in (0,1)$, $A$ an alphabet, $[\,,]$ a commutative bracket and  $\omega$ the intrinsic weight of $\hat{T}(A)$. Then we have following properties:
\begin{itemize}
\item \emph{(Lyons’ extension theorem)} for any $\gamma$-quasi-geometric rough path $X$ there exists a unique function $\mathbb{X}\colon [0,T]^2\to (\hat{T}(A))^*$ extending $X$ such that $\mathbb{X}$ satisfies  \eqref{eq:charachter_truncated_rough},  \eqref{eq:chen} and \eqref{eq:genrpbound_in} over all $\hat{T}(A)$. We call $\bX$ the Lyons extension of $X$.
\item \emph{(Lyons-Victoir's extension)} Under the condition $ 1\neq\gamma\omega(a)$ for any $a\in A$,  given a path $ x\colon [0,T]\to \bR^A$, $x=(x^a)_{a\in  A}$ satisfying
\begin{equation}\label{Lyons-Victoir_extension}
\sup_{a\in A}\sup_{s\neq t\in [0,T]^2}\frac{|x^a_t-x^a_s|}{|t-s|^{\gamma\omega(a)}}<+\infty\,,
\end{equation}
there exists a $\gamma$-quasi-geometric rough path $X$ over $x$. We call $X$ the Lyons-Victor extension of $x$.
\end{itemize}
\end{prop}
\begin{proof}
Both results can be easily proved via Theorem \ref{fund_thm_quasi} and some general properties of geometric rough paths. Given a  $\gamma$-quasi-geometric rough path $X$, we apply Lyons' extension Theorem (see \cite[Thm 2.2.1]{lyons1998}) to $\exp^*X$ obtaining the existence of a unique function  $\mathbb{X}'\colon [0,T]^2\to (T(A))^*$ extending $\exp^*X$ such that $\mathbb{X}'$ satisfies  \eqref{eq:charachter_truncated_rough},  \eqref{eq:chen} and \eqref{eq:genrpbound_in} over all $T(A)$. Using the same reasoning as in the proof of Theorem \ref{fund_thm_quasi} one has that $\mathbb{X}:=\log^* \mathbb{X}'$ satisfies the desired properties of existence and uniqueness. In case of the second property, for any given path $x$ as above, we can apply immediately the Lyons-Victoir extension theorem for geometric rough paths, in its reformulation given in \cite[Cor 4.10]{Tapia2018}, obtaining a $\gga$-weighted geometric rough path $Z$ over $x$, then the map $\log^*Z$ satisfies the desired properties. 
\end{proof}
\begin{rk}
Lyons's extension theorem and Lyons-Victoir's extension have been generalised,  when the rough path takes value over a generic Hopf algebra, see \cite{Tapia2018,Manchon2018}.
\end{rk}
Combining Proposition \ref{ext_lyo_quasi} with Theorem \ref{fund_thm_quasi}, we obtain that the bijection given by $\exp^*$ and $\log^*$ extends to the level of Lyons extension Theorem.
\begin{cor}\label{cor_exp}
We consider a $\gga$-quasi-geometric rough path $X$ and $\gga$-weighted geometric rough path $Y$ together with their corresponding Lyons Extension $\bX$ and $\bY$. Supposing that $Y=\exp^*X $  then  $\bY= \exp^*\bX$. Analogously, if $X=\log^*Y$ then $\bX= \log^*\bY$.
\end{cor}
\begin{proof}
The result follows trivially from the properties of $\exp$ and $\log$, together with the uniqueness of Lyons Extension on $\gga$-weighted geometric rough paths for any $\gga\in (0,1)$.
\end{proof}

\subsection{Examples of quasi-geometric rough paths}
We present two simple examples of quasi-geometric rough paths, where the quasi-shuffle structure arises naturally. These constructions represent an alternative to  branched rough paths in the description of objects not included in the class of geometric rough paths.

\subsubsection*{It\^o and Stratonovich rough paths over a Brownian motion}
The first example is obtained by rewriting the algebraic structure of the It\^o product $\diamond$ in the quasi-shuffle context and it represent a finite dimensional version of the algebraic structures contained in  \cite{kurusch15,Ebrahimi-Fard2015}.  Starting the definition of $A^{\circ}$ in the introduction and the recursive relation defining $\diamond$ in \eqref{defn_diamond}, we introduce the operation  $[\,, ]_{\diamond}\colon \bR^{A^{\circ}}\times \bR^{A^{\circ}}\to \bR^{A^{\circ}}$ given by any couple of letters
\[[a,b]_{\diamond}:= \mathbbm{1}_{A^{\circ}\setminus \{{\circ}\}}(a)\delta_{ab}\circ\,,\]
and extended linearly. It is trivial to see that $[,]_{\diamond}$ is a commutative bracket and by construction the corresponding quasi-shuffle product is exactly $\diamond$. Secondly, writing  $\bR^{A^{\circ}}= \bR^d\oplus \bR^{\{\circ\}}$ we obtain a decomposition  of $\bR^{A^{\circ}}$ satisfying property \eqref{decomposition_R_A}. We denote by $|\cdot|_{\diamond}$ the intrinsic weight  associated to it. A trivial computation shows the following identity 
\[T^2_{|\cdot|_{\diamond}}(A^{\circ})= \langle \1, i, \circ , ij \rangle\,,\quad  i,j\in\{1\,,\cdots \,, d\}\,. \]
We fix a $d$-dimensional standard Brownian motion $B=(B^1, \cdots, B^d)$ and we introduce the two parameters stochastic processes  $\mathbb{B}^I, \mathbb{B}^S\colon [0,T]^2\to (T^2_{|\cdot|_{\diamond}}(A^{\circ}))^*$ defined by the trivial condition $\langle\mathbb{B}^I_{st} ,\1\rangle= \langle\mathbb{B}^S_{st} , \1\rangle= 1$ and
\begin{equation}\label{def_Brownian_rough}
\begin{gathered}
\langle\mathbb{B}^I_{st} , i\rangle= \langle\mathbb{B}^S_{st} , i\rangle= B^i_t- B^i_s\,, \quad \langle\mathbb{B}^I_{st} , {\circ}\rangle= \langle\mathbb{B}^S_{st} , {\circ}\rangle= t-s\,,\\
\langle\mathbb{B}^I_{st} , ij\rangle= \int_{s}^t (B^i_r-B^i_s)dB^j_r\,, \quad \langle\mathbb{B}^S_{st} , ij\rangle= \int_{s}^t (B^i_r-B^i_s)d^{\circ}B^j_r\,,
\end{gathered}
\end{equation}
where we denote by $dB^j_r$ and $d^{\circ}B^j_r$ respectively the It\^o and  the Stratonovich integral as before. We can easily associate this couple of processes to a couple of rough paths.
\begin{prop}
For any $\gamma\in (1/3,1/2)$ conditions \eqref{def_Brownian_rough} identify with probability $1$ a $\gamma$-quasi-geometric rough path $\mathbb{B}^I $ and a $\gga$-weighted geometric rough path $\mathbb{B}^S$. We call them respectively the It\^o and Stratonovich rough paths. 
\end{prop}
\begin{proof}
We simply check that $\mathbb{B}^I$ and $\mathbb{B}^S$ satisfies a.s. the properties in Definition \eqref{def_inhom_rough_path}. Chen's property  immediately follows from the property that both It\^o and Stratonovich satisfy a.s. the additivity on intervals like the standard Lebesgue integration. Concerning the multiplicative property \eqref{eq:charachter_truncated_rough_in} this is equivalent to show for any $i,j\in\{1\,,\cdots \,, d\}$
\[\langle\mathbb{B}^I_{st} , i\rangle\langle\mathbb{B}^I_{st} , j\rangle= \langle\mathbb{B}^I_{st} , i\diamond j\rangle\,,\quad  \langle\mathbb{B}^S_{st} , i\rangle\langle\mathbb{B}^S_{st} , j\rangle= \langle\mathbb{B}^S_{st} , i\shuffle j\rangle\,,\]
which hold a.s. true because of the integration by parts identity in It\^o and Stratonovich calculus (see e.g. \cite{revuz2004continuous}). To prove the  the a.s. H\"older regularity of $\mathbb{B}^I$ and $\mathbb{B}^S$, we can easily check the hypothesis of Theorem \ref{Holder_criterion} on their components. Using the standard hypercontractivity estimates on It\^o iterated integrals (see e.g. \cite[Thm 2.7.2]{nourdin_peccati_2012}), for any $p\geq 1$ there exists a constant $C_p>0$ such that
\begin{equation}\label{eq_basis_induction}
\bE|\langle\mathbb{B}^I_{st} , i\rangle|^{2p}\leq C_p|t-s|^p\,,\quad \bE|\langle\mathbb{B}^I_{st} , ij\rangle|^{2p}\leq C_p (\bE|\langle\mathbb{B}^I_{st} , ij\rangle|^2)^{p}=  C_p2^{-p}|t-s|^{2p}\,.
\end{equation}
Therefore we satisfy the hypothesis of  Theorem \ref{Holder_criterion} with $N=2$ and we conclude that $\mathbb{B}^I$ is a quasi-geometric rough path. Since $\langle\mathbb{B}^S_{st} , ij\rangle$ is a linear combination of $\langle\mathbb{B}^I_{st} , ij\rangle$ and $(t-s)$, we conclude.
\end{proof}
A natural way to restate the results contained in \cite{gaines94} is then to link the processes $I$ and $S$ introduced in \eqref{def_I_and_S} with the rough paths $\mathbb{B}^I$ and $\mathbb{B}^S$. Since $\mathbb{B}^I$ and $\mathbb{B}^S$ are defined on $[0,T]^2$ and $I$ and $S$ are naturally defined on the $2$-simplex we can easily extend them by introducing the two parameters stochastic processes  $I, S\colon [0,T]^2\to (T(A^{\circ}))^*$ (we will adopt the same notation) defined as $\mathbb{B}^I$ and $\mathbb{B}^S$ on $\bR^{A^{\circ}}$ and given recursively as
\begin{equation}\label{def_higher_Brownian_rough}
\begin{gathered}
\langle I_{st} , w\circ\rangle= \int_s^t \langle I_{sr} , w\rangle dr\,, \quad \langle S_{st} , w\circ\rangle= \int_s^t \langle S_{sr} , w\rangle dr\,, \\
\langle I_{st} , wi\rangle= \int_s^t \langle I_{sr} , w\rangle dB^i_{r}\,, \quad \langle S_{st} , wi\rangle= \int_s^t \langle S_{sr} , w\rangle  d^{\circ}B^i_{r}\,. 
\end{gathered}
\end{equation}
Clearly  $I$ and $S$ extend the processes $I$ and $S$ introduced in \eqref{def_I_and_S} and the processes $\mathbb{B}^I$ and $\mathbb{B}^S$ over a wider set of words. This second extension can be restated in terms of rough paths too.
\begin{thm}\label{thm_lyons_ext_rough}
The functions $I$ and $S$ defined by \eqref{def_higher_Brownian_rough} coincide a.s. with  the Lyons extension of the rough paths  $\mathbb{B}^I$ and $\mathbb{B}^S$.
\end{thm}
\begin{proof}
By construction $I$ and $S$ extend respectively $\mathbb{B}^I$ and $\mathbb{B}^S$. Thus the result hold if we are able to prove that $I$ and $S$ satisfy a.s. the properties of Definition \ref{def_inhom_rough_path} over all $\hat{T}(A)$ and $T(A)$. First of all, Chen's relations are automatically satisfied because, as recalled before, both It\^o and Strotonovich integrals are additive on on intervals. Secondly, we deduce immediately from \cite[Prop. 2.2, Prop 2.3]{gaines94}, which shows the properties \eqref{charachter_brownian}, that one has a.s. the property  \eqref{eq:charachter_truncated_rough_in} on $I_{st}$ and $S_{st}$ for any $s<t$ (it is sufficient to repeat the same proof starting from $s$ and not $0$). Chen's property and the character property imply  that for any $v<u$ $I_{uv}$ and $S_{uv}$ are a.s. the inverse of characters $I_{vu}$, $S_{vu} $ in the group of characters. Thus we have the a.s. identity
\begin{equation}\label{id_inverse_antipode}
S_{uv}= S_{vu}\cA_{\shuffle}\,, \quad I_{uv}= I_{vu}\cA_{\diamond}\,, 
\end{equation}
where $\cA_{\shuffle}$ and $\cA_{\diamond}$ are respectively the antipode maps of the shuffle and quasi-shuffle Hopf algebra. Since $\cA_{\shuffle}$ and $\cA_{\diamond}$ are both algebra homomorphisms we conclude that $I$ and $S$ satisfy  \eqref{eq:charachter_truncated_rough_in} for any $s,t\in [0,T]^2$. The final H\"older relations are satisfied using Theorem \ref{Holder_criterion} in case of a countable family of estimates (as explained in  Remark \ref{Holder_criterion_infinite}) by showing that for any $w\in W(A^{\circ})$ and $p\geq 1$ there exists two constants $C_p(w), C'_p(w)>0 $ depending on $p$ and $w$ such that one has for any $s,t\in [0,T]^2$
\begin{equation}\label{final_estimates}
\bE |\langle I_{st} , w\rangle|^{2p}\leq C_p(w)|t-s|^{|w|_{\diamond}p}\,,\quad \bE|\langle S_{st} , w\rangle|^{2p}\leq C'_p(w)|t-s|^{|w|_{\diamond}p}
\end{equation}
We prove these estimate by induction on the word length. The basis of induction is given trivially by essentially by the estimates in \eqref{eq_basis_induction} and the definitions of $\bB^I$ and $\bB^S$. Supposing these estimate valid on every word $w$ of length $n $ and we will prove \eqref{final_estimates} on $wa$ for some $a\in A^{\circ}$. In case $a= \circ$ we apply the recursive definition in  \eqref{def_higher_Brownian_rough} together with the classical Jensen inequality and Fubini Theorem  to obtain
\[\begin{split}\bE |\langle I_{st} , w\circ\rangle|^{2p}&=  \bE\left|\int_s^t \langle I_{sr} , w\rangle dr\right|^{2p} \leq |t-s|^{2p-1} \int_s^t \bE|\langle I_{sr} , w\rangle|^{2p} dr
\end{split}\]
Applying the induction hypothesis and performing the elementary integration, we conclude that there exists a new constant $M>0$ such that
\begin{equation}\label{proof_first_Holder}
\bE |\langle I_{st} , w\circ\rangle|^{2p}\leq M |t-s|^{2p-1}  |t-s|^{|w|_{\diamond}p+1}=  M  |t-s|^{|w\circ|_{\diamond}p}\,. 
\end{equation}
Same reasoning for the quantity $\bE |\langle S_{st} , w\circ\rangle|^{2p}$. In case $a=i\in\{1,\cdots , d\}$ we apply the standard BDG inequality, see \cite{revuz2004continuous},  to $\langle I_{st} , wi\rangle$ obtaining that there exists a constant  $C_p>0$ depending on $p$ such that
\[\bE |\langle I_{st} , wi\rangle|^{2p}= \bE\left|\int_s^t \langle I_{sr} , w\rangle dB^i_r\right|^{2p}\leq  C_p \left|\int_s^t \bE\left|\langle I_{sr} , w\rangle\right|^2 dr \right|^{p}\,.\]
We apply again the induction hypothesis, concluding that there exists a new constant $M'>0$ such that
\begin{equation}\label{proof_second_Holder}
\bE |\langle I_{st} , wi\rangle|^{2p}\leq M'   |t-s|^{p(|w|_{\diamond}+1)}=  M  |t-s|^{p|wi|_{\diamond}}\,. 
\end{equation}
In case of $\bE |\langle S_{st} , wi\rangle|^{2p}$, we apply the standard It\^o-Stratonovich correction for semimartingales and the convexity of $y\to |y|^{2p}$ obtaining
\[\begin{split}\bE |\langle S_{st} , wi\rangle|^{2p}&= \bE \left|\int_s^t \langle S_{sr} , w\rangle dB^i_r+ \frac{\delta_{w_ni}}{2} \int_s^t \langle S_{sr} , w_1\cdots w_{n-1}\rangle dr\right|^{2p}\\&\leq 2^{2p-1}\left(\bE \left|\int_s^t \langle S_{sr} , w\rangle dB^i_r\right|^{2p}+ \frac{\delta_{w_ni}}{2^{2p}} \bE \left| \langle S_{st} , w_1\cdots w_{n-1}\circ\rangle  \right|^{2p}\right)\,. \end{split}\]
Applying again the BDG inequality as before an the recursive hypothesis on $w$ and $w_1\cdots w_{n-1}\circ $ we obtain that there exists a constant $M''>0$ such that
\begin{equation}\label{proof_third_Holder}
\bE |\langle S_{st} , wi\rangle|^{2p}\leq M''  \left( |t-s|^{p(|w|_{\diamond}+1)}+ |t-s|^{p(|w_1\cdots w_{n-1} |_{\diamond}+2)}\right)=  2M''  |t-s|^{p|wi|_{\diamond}}\,. 
\end{equation}
Combining the estimates \eqref{proof_first_Holder}, \eqref{proof_second_Holder} and \eqref{proof_third_Holder} we complete the induction and we obtain the final estimates \eqref{final_estimates}, thereby yielding the result.
\end{proof}
An immediate consequence of this result is an alternative proof of the relations between the functions $I$ and $S$, originally stated in  \cite{Arous1989}.
\begin{cor}\label{first_main_thm}
One has the a.s. identities $\exp^*I= S$ and $\log^*S= I$, which becomes for any word $w\in W(A^{\circ})$
\begin{equation}\label{explicit_form_brow}
\begin{gathered}\langle S_{st} , w\rangle= \langle I_{st} , w\rangle+ \sum_{u\in [[w]]}\frac{1}{2^{|w|-|u|}}\langle I_{st} , u\rangle\,,\\ \langle I_{st} , w\rangle=  \langle S_{st} , w\rangle + \sum_{u\in [[w]]}\frac{(-1)^{|w|-|u|}}{2^{|w|-|u|}}\langle S_{st} , u\rangle\,, \end{gathered}
\end{equation}
where the set $[[w]]$ consists of the words we can construct from $w$ by successively replacing  with  $\circ$ any neighbouring pairs of letters of the form $ii$ for some $i \in \{1\cdots d\}$.
\end{cor}
\begin{proof}
Thanks to explicit definitions of $\exp$ and $\log$ in \eqref{explicit_exp_log} on two letters, together with the standard It\^o-Stratonovich correction for semimartingales, we have trivially $ \exp^*\mathbb{B}^I= \mathbb{B}^S$ and $\log^*\mathbb{B}^S= \mathbb{B}^I $. Therefore the desired identities are an immediate consequence of Corollary \ref{cor_exp} and Theorem \ref{thm_lyons_ext_rough}. The explicit formulae are then consequence of the general definition of $\exp$ and $\log$ in \eqref{prop_exp}.
\end{proof}

\begin{rk}
The first identitity in \eqref{explicit_form_brow} is a specific case of a general identity in stochastic analysis between iterated Stratonovich integral and iterated It\^o integrals of semimartingales, see \cite{Hu1988,Arous1989}. A direct reinterpretation of this general identity using  quasi-shuffle algebra has been carried out in \cite{kurusch15}.
\end{rk}
\subsubsection*{Standard theory of rough paths in case $\gamma\in (1/3, 1/2)$}
The second example we consider is a rewriting of the usual definition of rough paths when $\gamma\in (1/3, 1/2)$, as described in \cite{Friz2020course}.  In this particular case, most general definitions simplify drastically and we can define an elementary $\gga$-rough path as a couple of objects $X=(x, \bX)$, where $x\colon [0,T]\to \bR^d $, $x=(x^1\,, \cdots\,, x^d)$ is a $\gamma$-H\"older path  and $\bX\colon [0,T]^2\to \bR^d\otimes \bR^d$, $\bX=(\bX^{ij}\colon  i,j\in\{1\,,\cdots \,, d\})$ is a $2\gamma$-H\"older function satisfying for any $s,u,t\in [0,T]$ and $i,j\in\{1\,,\cdots \,, d\} $ the algebraic identity
\begin{equation}\label{eq_simple_chen}
\bX^{ij}_{st}= \bX^{ij}_{su}+ \bX^{ij}_{ut}+ (x^i_u-x^i_s) (x^j_{t}-x^j_u)\,.
\end{equation}
Even in this case we can describe every $\gamma$-rough path via a geometric and quasi-geometric structures like before. To introduce them, we fix an integer parameter $d\geq 1$ and we consider the alphabet
\begin{equation}\label{alphabet_ito_extension}
\bA_2^d= \{v\in \bN^{d}\setminus \{0\} \colon  \sum_{i=1}^d v_i\leq 2\}\,.
\end{equation}
Multi-indexes $v\in \bA_2^d$ have an intrinsic commutative operation of sum between them. We use this operation to define the following one $ [\,, ]_{2}\colon \bR^{\bA_2^d}\times \bR^{\bA_2^d}\to \bR^{\bA_2^d}$, given by
\begin{equation}
[\ga, \gb]_2: =\left\{
	\begin{array}{ll}
	\ga +\gb  & \mbox{if } \ga+\gb\in \bA_2^d\,, \\
		0 & \mbox{otherwise }
	\end{array}
	\right.
\end{equation}
and extended linearly. As before, we can check trivially that $[,]_{2}$ is a commutative bracket. We denote by $\qshuffle_2$ the associate quasi shuffle product on it. Secondly, we partition $\bA_2^d$ as
\[\bA_2^d= \{e_i\colon i\in\{1\,,\cdots \,, d\}\}\sqcup \{[e_ie_j]\colon  i,j\in\{1\,,\cdots \,, d\}\}= A_1\sqcup A_2\,,\]
where $e_i$ is $i$-th element of the canonical basis in $\bR^d$. Writing $\bR^{\bA_2^d}= \bR^{A_1}\oplus \bR^{A_2}$ we obtain again a decomposition  of $\bR^{\bA_2^d}$ where we can define a intrinsic weight, which we denote in this case by $|\cdot|_{2}$. Using the same notation with the canonical basis, it is straightforward to show that 
\[T^2_{|\cdot|_{2}}(\bA_2^d)= \langle \1, e_i, e_j+e_i, e_je_i \rangle\,,\quad  i,j\in\{1\,,\cdots \,, d\}\,. \]
The specific structure imposed by the hypothesis $\gamma\in (1/3,1/2)$ allows to extend every $\gamma$-rough path $X=(x, \bX)$, to a a geometric and a quasi-geometric rough path with the ``same algebraic properties" as $\bB^I$ and $\bB^S$.
\begin{prop}\label{prop_ito_and_strato_ext}
For any $\gamma\in (1/3, 1/2)$ and every elementary $\gga$-rough path $X=(x, \bX)$, we introduce the functions $X^I, X^{S}\colon [0,T]^2\to (T^2_{|\cdot|_{2}}(\bA_2^d))^*$  defined by the  conditions  
\begin{equation}\label{def_rough_Ito_and_strato}
\begin{gathered}
\langle X^I_{st} ,\1\rangle= \langle X^S_{st} , \1\rangle= 1, \quad \langle X^I_{st} , e_i\rangle= \langle X^S_{st} , e_i\rangle= x^i_t- x^i_s\,, \\
 \langle X^I_{st} , [e_ie_j]\rangle=\langle X^S_{st} , [e_ie_j]\rangle=(x^i_t-x^i_s) (x^j_{t}-x^j_s) - \bX^{ij}_{st}- \bX^{ji}_{st}\,, \\
\langle X^I_{st} , e_ie_j\rangle=\bX^{ij}_{st}\,, \quad \langle X^S_{st} , e_ie_j\rangle= \bX^{ij}_{st} -\frac{1}{2}\langle X^I_{st} , [e_ie_j]\rangle\,. 
\end{gathered}
\end{equation}
Then $X^{I}$ is a $\gamma$-quasi-geometric rough path  and $X^S$ is a $\gga$-weighted geometric rough path. We call them the It\^o and Stratonovich extension of $X$.
\end{prop}
\begin{proof}
By hypothesis on $X$, the functions $X^{I}$ and $X^{S}$ are well-defined and they have the right H\"older regularities. Concerning the multiplicative property, we can easily show the identities
 \[ \langle X^I_{st} , e_i\qshuffle_2 e_j\rangle= \langle X^I_{st} , e_i\rangle\langle X^I_{st} ,  e_j\rangle\,, \quad \langle X^S_{st} , e_i\shuffle e_j\rangle= \langle X^S_{st} , e_i\rangle\langle X^S_{st} ,  e_j\rangle\,,\]
by rearranging the properties defining $X^{I}$ and $X^{S}$. The last property to check are Chen's relations and the only non trivial identity to prove is the equality
\begin{equation}\label{addtive_bracket}
\langle X^I_{st} , e_i+ e_j\rangle= \langle X^I_{su} , [e_ie_j]\rangle+ \langle X^I_{ut} , e_i+ e_j\rangle\,,
\end{equation}
for any $s,u,t\in [0,T]$ and $i,j\in\{1\,,\cdots \,, d\}$.  But this last identity follows as an elementary consequence of  \eqref{eq_simple_chen}.
\end{proof}
\begin{rk}
It is straightforward to see that if we consider a  a $d$ dimensional Brownian motion $B$ endowed By the family of its Lévy areas $\bB^{ij}= \int_{s}^t (B^i_r-B^i_s)dB^j_s  $, the  constructions of  $X^I$ and $X^S$ coincide essentially  with the definition of $\bB^I$ and $\bB^S$ in \eqref{def_Brownian_rough}.
\end{rk}
\begin{rk}
A direct consequence of the identity \eqref{addtive_bracket} implies that there exist a  unique family of paths $\{[X]^{ij}\}_{i,j\in \{1, \cdots, d\}}$ such that  $[X]^{ij}_0=0$ and $\langle X^I_{st} , e_i+ e_j\rangle= [X]^{ij}_t- [X]^{ij}_s$ for any $i,j\in \{1, \cdots, d\}$. This family of paths is known  in the literature as the \emph{bracket of a rough path $X$} (see \cite[Defn 5.5]{Friz2020course}) and it provides an analytic substitute of the notion of quadratic variation in the rough path setting. We remark also that if a semimartingale  $y=(y^1\,, \cdots\,, y^d)$ and its iterated It\^o integrals $\bY^{ij}= \int_{s}^t (y^i_r-y^i_s)dy^j_s $ are a.s.  a $\gamma$-rough path $Y=(y, \bY)$, the bracket of $Y$ is exactly the quadratic variation of $y$.
\end{rk}

\section{Rough change of variable formula}\label{second_section}
We consider now the main application of quasi-geometric rough paths: the deterministic change of variable formula for branched rough paths.  In what follows, we fix two parameters $\gamma\in (0, 1)$, $d\geq 1$ integer and we  suppose given a  $\gga$-H\"older path  $x\colon [0,T]\to \bR^{d}$, $x=(x^1,\cdots, x^d)$  such that there exists a $\gamma$-branched rough path $\mathbf{X}\colon [0,T]^2\to \cH(\{1\,,\cdots\,, d\})$  over $x$. By construction of $\X$, the following path  $ X\colon [0,T]\to (\cH(\{1\,,\cdots\,, d\}))^{d}$, $X=\{X^{i}\colon [0,T]\to \cH(\{1\,,\cdots\,, d\})\}_{i\in\{1\,, \cdots,d\}}$ defined by 
\begin{equation}\label{coordinate2}
\langle \gs,X^i_t\rangle= \begin{cases} x_t^i & \text{if}\;\gs=\1^*\,,\\
1 & \text{if  $\gs=\bullet_i$,}\\ 0 & \text{for any other $\gs\in \cF(A)$} \,,
\end{cases}
\end{equation} 
satisfies immediately the properties of controlled rough path, obtaining $X\in \cD^{N\gamma}_{\mathbf{X}}$ for any $N\geq 1$. Since  we can apply the operations of composition and rough integration on the controlled path $X$, we ask ourselves if the branched rough integral satisfies an identity similar to the standard change of variable formula \eqref{classic_change}, meaning that for any sufficiently smooth function $\gp\colon \bR^d\to \bR$ one has the identity
\begin{equation}\label{first_eq_chg}
\gp(x_t)- \gp(x_s)= \sum_{i=1}^d\int_s^t \partial_{i}\Phi(X_r)d^b\X^i_r\,,
\end{equation}
where $\partial_i\Phi(X_r)$ is the controlled rough path obtained by composition of $X$ in \eqref{coordinate2} with $\partial_{x_i}\gp$ as defined in Proposition \eqref{composition}. In case $\gamma>1/2$, we deduce from the definition \eqref{defn_rough_int_br} that for any $i\in A$ one has
\[\int_0^t \partial_i\Phi(X_r)d^b\X^i_r= \lim_{|\pi|\to 0}\sum_{[s,u]\in\pi}\,\langle \1^*,\partial_i\Phi(X_s)\rangle \langle \X_{su},\bullet_i \rangle= \lim_{|\pi|\to 0}\sum_{[s,u]\in\pi}\,\partial_{x_i}\gp(x_s)(x^i_u-x^i_s)\]
Then the branched rough integral coincide with the  Young integral between  $\partial_{x_i}\gp(x)$ and $x^i$ (see \cite{Young1936}) and the identity \eqref{first_eq_chg} holds from  standard results on Young integration (see \cite{Friz2020course} for further details). However, in the general case $\gamma\leq 1/2$  formula \eqref{first_eq_chg} might not hold in general. The general problem of the branched change of variable formula consists in finding some suitable hypothesis to correct \eqref{first_eq_chg} with the increment of a function, that is to prove the existence of a  function $R^{\gamma,d, \gp}\colon [0,T]\to \bR$ such that  one has 
\begin{equation}\label{second_eq_chg}
\gp(x_t)- \gp(x_s)= \sum_{i=1}^d\int_s^t \partial_i\Phi(X_r)d^b\X^i_r +R^{\gamma,d, \gp}_t- R^{\gamma,d, \gp}_s\,.
\end{equation}
This problem has been deeply studied in the Phd thesis by David Kelly \cite{kelly2012ito} via the key notion of \emph{simple bracket extension}. In what follows, we recall the main ideas behind this notion and we will show how quasi-geometric rough paths provide an efficient way to construct effective examples of simple bracket extensions.

\subsection{Kelly's change of variable formula}
In order to write down a formula like  \eqref{second_eq_chg}, we need to rewrite  the expansion of the branched rough integral, as  explained in the \eqref{equ:rough_int_est} combined  with the Taylor expansion of $\gp(x)$, so that it is possible to replace them with the increment of a function $R^{\gamma,d, \gp}$ up to an order $o(|t-s|)$. The main idea contained in \cite{kelly2012ito} is that  we can choose the function  $R^{\gamma,d, \gp}$ as a sum of rough integrals, provided  there exists a specific branched rough path  $\widehat{\X}$ defined on a wider alphabet  containing $\{1,\cdots,d\}$ and extending $\X$. Let us introduce the basic definitions to define properly this underlying alphabet.
\begin{defn}
For any $\gga\in (0,1)$ and $d\geq 1$ integer we define the  alphabet $\cA_{\gamma}^d$ as
\[\cA^d_{\gamma}=\bigcup_{n=1}^{N_{\gamma}}\{1, \cdots, d\}^n\,,\]
where $N_{\gamma} $ is defined as before. The elements belonging to $ \{1,\cdots, d\}^{n}$ with $n\geq 2$ are denoted as $(i_1\cdots i_n)$, where $i_1,\cdots, i_n\in \{1,\cdots, d\}$ and the elements $i \in \{1,\cdots, d\}$ are denoted by $i$.
\end{defn}
We trivially remark that $\{1,\cdots, d\} $ is embedded in $\cA_{\gamma}^d$  for any $\gga\in(0,1)$ and we have  an equality if $1/2<\gga<1$. Starting from  this alphabet we  consider $\cH(\cA^d_{\gamma})$, the Butcher-Connes-Kreimer Hopf algebra built from $\cA_{\gamma}^d$, which  contains $\cH(\{1,\cdots, d\})$ as a sub  Hopf algebra. The intrinsic structure of $\cA_{\gamma}^d$ allows to define a function $|\cdot|_{\gamma}\colon \cF(\cA_{\gamma}^d)\to \bN $ given recursively by the identities
\[ |B_{(i_1\cdots i_{n})}^+(\t_1\cdots \tau_k)|_{\gamma}= n+ |\t_1\cdots \tau_k|_{\gamma}\,,\quad |\t_1\cdots \tau_k|_{\gamma}= \sum_{j=1}^k|\t_j|_{\gamma}\,, \quad |\1|_{\gamma}=0\,. \]
We can easily show that  $|\cdot|_{\gamma}$ is a compatible weight for $\cH(\cA_{\gamma}^d)$ endowed with the forest cardinality $|\cdot|$. We call it the \emph{intrinsic weight of $\cH(\cA_{\gamma}^d)$}. By construction of the intrinsic weight, for any $n\leq N_{\gamma}$ the grafting maps $B^+_{(i_1\cdots i_{n})} $ are integration maps of order $n$ for $|\cdot|$. The alphabet  $\cA_{\gamma}^d$ has also   a natural notion of symmetry among his elements.
\begin{defn}
Two elements of $a,b\in \cA_{\gamma}^d$, $a=(i_1\cdots i_n)$, $b=(j_1\cdots j_m)$ are said to be \emph{symmetrical} if $m=n$ and there exists a permutation $\gs\in S_{n}$ such that $a_{\gs}= (i_{\gs(1)}\cdots i_{\gs(n)})=b$. We denote it by the symbol $a\sim b$.
\end{defn}
It is straightforward to check that $\sim$ is an equivalence relation over the set $\cA_{\gamma}^d$. Moreover we can easily extend it to an equivalence relation on the set of forests $ \cF(\cA_{\gamma}^d)$ and we denote it by the same notation.
Using the letters of the alphabet $\cA_{\gamma}^d$, we can now define a specific linear combination of forests: the \emph{bracket polynomials}.
\begin{defn}
Let $\gga\in (0,1)$ and  $(i_1 \cdots i_n)\in \cA_{\gamma}^d$. We define \emph{the bracket polynomial} $\llangle i_1 i_2\cdots i_n\rrangle\in \cH(\cA_{\gamma}^d)$ as
\begin{equation}\label{bracket_pol}
\llangle i_1 i_2\cdots i_n\rrangle = \bullet_{i_1} \ldots \bullet_{i_n}- \sum_{ \{a,b\}=\{i_1\,,\cdots, i_n\}}B_{(b_1\cdots b_{k})}^+(\bullet_{a_1} \ldots \bullet_{a_{n-k}})\,,
\end{equation}
where the sum is done over all ways of splitting the set $\{i_1\,,\cdots, i_n\}=\{a,b\}$ into two non-empty sets $a=\{a_1,\cdots,a_{n-k}\}$ and  $ b=\{b_1,\cdots, b_{k}\}$ for every $k\geq 1$.
\end{defn}
By definition  $ \llangle i \rrangle= \bullet_{i}$ for any $i\in \{1, \cdots, d\}$. Thus we can expect that the bracket polynomial $\llangle  i_1 i_2\cdots i_n\rrangle$ is related with the tree $\bullet_{ (i_1 i_2\cdots i_n)}$. Let us recall the two main properties of the bracket polynomial. Both  results as well as their proofs are proven in several parts of \cite[Chap. 5]{kelly2012ito} and we decide to encode them in a unique proposition.
\begin{prop}
Let $\gga\in (0,1)$, $d\geq 1$ integer and  $(i_1 \cdots i_n)\in \cA_{\gamma}^d$. One has  the following properties:
\begin{enumerate}[label=\alph*\emph{)}]
 \item For all $a \in \cA_{\gamma}^d $ such that $a\sim (i_1 \cdots i_n)$ one has $\llangle a\rrangle= \llangle i_1 \cdots i_n\rrangle$.

\item Let  $\mathbf{B}^+\colon\cH(\{1,\cdots, d\})\otimes \cH(\{1,\cdots, d\}) \rightarrow \cH(\cA_{\gamma}^d)$ be the unique linear map defined for any couple of forests $\gs, \tau\in \cH(\{1,\cdots, d\})$ as
\[\mathbf{B}^+(\gs\otimes \t):= \begin{cases} B_{(i_1\cdots i_n)}^+(\gs)& \text{if $\t=\bullet_{i_1} \ldots \bullet_{i_n}$ for any $1 \leq n \leq N$}\\ 0 &\text{otherwise .}
\end{cases}\]
Then we have the identity 
\begin{equation}\label{defn_alternative_bracket}
\llangle i_1 \cdots i_n\rrangle :=\bullet_{i_1} \ldots \bullet_{i_n}- \mathbf{B}^+\gD'(\bullet_{i_1}\cdots \bullet_{i_n})\,,
\end{equation}
where $\gD'$ is the reduced coproduct \eqref{defn_coproduct} associated to the coproduct $\gD$ in  \eqref{defn_BCK_coproduct}.
\item Applying the coproduct $\gD\colon\cH(\cA_{\gamma}^d)\otimes \cH(\cA_{\gamma}^d) \rightarrow \cH(\cA_{\gamma}^d)$ to $\llangle i_1 i_2\cdots i_n\rrangle$, we have
\begin{equation}\label{delta_bracket}
\begin{split}\gD\llangle i_1 i_2\cdots i_n\rrangle&= \llangle i_1 i_2\cdots i_n\rrangle\otimes\1 +\1\otimes \llangle i_1 i_2\cdots i_n\rrangle \\&+ \sum_{ \{a,b\}=\{i_1\,,\cdots, i_n\}}\bullet_{a_1} \ldots \bullet_{a_{n-k}}\otimes(\llangle  b_1\cdots b_{k} \rrangle  - \bullet_{(b_1\cdots b_{k})})\,,\end{split}
\end{equation}
where the sum is done over the same set as in \eqref{bracket_pol}.
\end{enumerate}
\end{prop}
\begin{proof}
We will recall the proof of the three properties for sake of completeness. The proof of $a)$ follows from the fact that for any permutation $\gs\in S_n$ one has  $\llangle i_{\gs(1)}\cdots i_{\gs(n)}\rrangle=\llangle i_1\cdots i_n\rrangle$ because the two terms defining $\llangle i_1\cdots i_n\rrangle$ in  \eqref{bracket_pol} are the same  when we permute the indices with $\gs$. To prove $b)$, we can iterate the multiplicative property of $ \gD$ to obtain the formula
\[\gD(\bullet_{i_1}\cdots \bullet_{i_n})=\bullet_{i_1}\cdots \bullet_{i_n}\otimes \1+ \1\otimes\bullet_{i_1}\cdots \bullet_{i_n}+\sum_{ \{a,b\}=\{i_1\,,\cdots, i_n\}}\bullet_{a_1} \ldots \bullet_{a_{n-k}}\otimes \bullet_{b_1} \ldots \bullet_{b_k}\,,\]
which can be easily proved by induction. Thus by definition of  reduced coproduct in  \eqref{defn_coproduct} we obtain 
\begin{equation}\label{key_combinatorial}
\gD'(\bullet_{i_1}\cdots \bullet_{i_n})= \sum_{ \{a,b\}=\{i_1\,,\cdots, i_n\}}\bullet_{a_1} \ldots \bullet_{a_{n-k}}\otimes\bullet_{b_1} \ldots \bullet_{b_k}
\end{equation}
and using the operator $\mathbf{B}^+$ we obtain the identity \eqref{defn_alternative_bracket}. In order to conclude the proof with  the identity $c)$, We combine the definition of the reduced coproduct in \eqref{defn_coproduct} to obtain
\[
\begin{split}\gD\llangle i_1 i_2\cdots i_n\rrangle&= \llangle i_1 i_2\cdots i_n\rrangle\otimes\1 +\1\otimes \llangle i_1 i_2\cdots i_n\rrangle + \gD'\llangle i_1 i_2\cdots i_n\rrangle\,.
\end{split}
\]
Using the following identity  for any $b\in \cA_{\gamma}^d$ and $\gs\in \cF( \cA_{\gamma}^d)$
\[ \gD'(B^{+}_b(\gs))= (\text{id}\otimes B_{b}^+)\gD'\gs +  \gs\otimes \bullet_b\,, \] 
We combine the definition of $\mathbf{B}^+$ with \eqref{key_combinatorial} and the coassociativity of $\gD'$ in \eqref{eq:coassociative_reduced_coprod} to compute
\[\begin{split}
&\gD'\sum_{ \{a,b\}=\{i_1\,,\cdots, i_n\}}B_{(b_1\cdots b_{k})}^+(\bullet_{a_1} \ldots \bullet_{a_{n-k}})=\\&= \sum_{\{a,b\}=\{i_1\,,\cdots, i_n\}} (\text{id}\otimes B_{(b_1\cdots b_{k})}^+)\gD'(\bullet_{a_1} \ldots \bullet_{a_{n-k}})+ \bullet_{a_1} \ldots \bullet_{a_{n-k}}\otimes \,\bullet_{(b_1\cdots b_{k})}\\& =\sum_{\{a,b\}=\{i_1\,,\cdots, i_n\}} (\text{id}\otimes \mathbf{B}^+)(\gD'\otimes \text{id})(\bullet_{a_1} \ldots \bullet_{a_{n-k}}\otimes   \bullet_{b_1} \ldots \bullet_{b_k})+ \bullet_{a_1} \ldots \bullet_{a_{n-k}}\otimes\, \bullet_{(b_1\cdots b_{k})}\\& = (\text{id}\otimes\mathbf{B}^+)(\gD'\otimes \text{id})\gD'(\bullet_{i_1}\cdots \bullet_{i_n})+\sum_{\{a,b\}=\{i_1\,,\cdots, i_n\}} \bullet_{a_1} \ldots \bullet_{a_{n-k}}\otimes \,\bullet_{(b_1\cdots b_{k})}\\&= (\text{id}\otimes\mathbf{B}^+)(\text{id}\otimes \gD')\gD'(\bullet_{i_1}\cdots \bullet_{i_n})+\sum_{\{a,b\}=\{i_1\,,\cdots, i_n\}} \bullet_{a_1} \ldots \bullet_{a_{n-k}}\otimes\, \bullet_{(b_1\cdots b_{k})}\\&= \sum_{\{a,b\}=\{i_1\,,\cdots, i_n\}} \bullet_{a_1} \ldots \bullet_{a_{n-k}} \otimes\, \left( \mathbf{B}^+( \gD'(\bullet_{b_1} \ldots \bullet_{b_k})) + \bullet_{(b_1\cdots b_{k})}\right) \,.
\end{split}\]
Using again the combinatorial identity \eqref{key_combinatorial} and the alternative definition \eqref{defn_alternative_bracket} one has
\[\begin{split}
& \gD'\left(\bullet_{i_1} \ldots \bullet_{i_n}-\mathbf{B}^+\gD'(\bullet_{i_1}\cdots \bullet_{i_n})\right)=\\&=\sum_{ \{a,b\}=\{i_1\,,\cdots, i_n\}}\bullet_{a_1} \ldots \bullet_{a_{n-k}}\otimes \left(\bullet_{b_1} \ldots \bullet_{b_k}-\mathbf{B}^+( \gD'(\bullet_{b_1} \ldots \bullet_{b_k}))- \bullet_{(b_1\cdots b_{k})}\right)\\&=\sum_{ \{a,b\}=\{i_1\,,\cdots, i_n\}}\bullet_{a_1} \ldots \bullet_{a_{n-k}}\otimes \left(\llangle  b_1\cdots b_{k} \rrangle- \bullet_{(b_1\cdots b_{k})}\right)\,. \end{split}\]
Thereby obtaining the thesis.
\end{proof}

The identity \eqref{delta_bracket} has a very deep consequence. Indeed for any fixed $\gamma$-rough path $\widehat{\X}$ over $\cH(\cA_{\gamma}^d)$  and any triplet $s,u,t\in [0,T]$, the formula  \eqref{delta_bracket} together with the Chen's identity  of $\widehat{\X}$ implies
\[\begin{split}
\langle\widehat{\X}_{st},\llangle i_1 i_2\cdots i_n\rrangle\rangle&=\langle\widehat{\X}_{su},\llangle i_1 i_2\cdots i_n\rrangle\rangle+\langle\widehat{\X}_{ut},\llangle i_1 i_2\cdots i_n\rrangle\rangle\\&+ \sum_{ \{a,b\}=\{i_1\,,\cdots, i_n\}}\langle\widehat{\X}_{su},\bullet_{a_1} \ldots \bullet_{a_{n-k}} \rangle\langle\widehat{\X}_{ut},\llangle  b_1\cdots b_{k} \rrangle  - \bullet_{(b_1\cdots b_{k})}\rangle\,.
\end{split}\]
Therefore the function  $(s,t)\to \langle\widehat{\X}_{st},\llangle i_1 i_2\cdots i_n\rrangle\rangle$ may be reinterpreted as the increment of a path as long as $\langle\widehat{\X}_{st},\llangle  b_1\cdots b_{k} \rrangle \rangle= \langle\widehat{\X}_{st},\bullet_{(b_1\cdots b_{k})}\rangle$ for any letter $(b_1\cdots b_{k})$ with $k<n$. This observation is at the basis of the  notion of simple bracket extension. 

\begin{defn}\label{bracket}
Let $\gga\in (0,1)$ and $d\geq 1$ integer. Every  $\gga$-weighted rough path $\widehat{\X}$ over $\cH(\cA_{\gamma}^d)$ with respect to its intrinsic weight is said to be a \emph{simple bracket extension} if it satisfies the following properties:
\begin{enumerate}[label={\arabic*)}]
\item for all $i_1, \cdots, i_n \in \{1,\cdots, d\}$, $1 \leq n \leq N$, one has the identity 
\begin{equation}\label{Ibpf}
\langle\widehat{\X}_{st},\bullet_{(i_1\cdots i_n)}\rangle= \langle\widehat{\X}_{st},\llangle i_1 \cdots i_n\rrangle\rangle
\end{equation}
for any $s,t\in [0,T]$;
\item For any couple of forests $\gs,\t\in \cF(\cA_{\gamma}^d)$ such that $\t\sim \gs$ then $\langle\widehat{\X}_{st}, \t \rangle=  \langle\widehat{\X}_{st}, \gs \rangle$.
\end{enumerate}
If $\X$ is a branched rough path over $\cH(\{1,\cdots,d\})$ such that  $\langle\X_{st}, \tau\rangle =\langle\widehat{\X}_{st}, \tau \rangle $ for any $\tau \in \cH(\{1,\cdots,d\})\subset \cH(\cA_{\gamma}^d)$, we say that $\widehat{\X}$ is a simple bracket extension over $\X$.
\end{defn}
\begin{rk}
We immediately remark that this notion of simple bracket extension coincide with the usual notion of branched rough path when $\gamma>1/2$. An informal way to describe a simple bracket extension over a branched rough path $\X$ when $\gamma\leq 1/2$ is then to add a family of extra paths $\widehat{x}=\{\widehat{x}^{(i_1\cdots i_n)}\}_{(i_1\cdots i_n)\in \cA_{\gamma}^d} $ and a family of functions $\{\langle\widehat{\X}_{st},B_{(i_1\cdots i_n)}^+(\cdot)\rangle\}_{(i_1\cdots i_n)\in \cA_{\gamma}^d}$ satisfying
\begin{equation}\label{pol_Ibpf}
\langle \X_{st}, \bullet_{i_1} \ldots \bullet_{i_n}\rangle=\sum_{\{a,b\}=\{i_1\,,\cdots, i_n\}}\langle \widehat{\X}_{st},B_{(b_1\cdots b_{k})}^+(\bullet_{a_1} \ldots \bullet_{a_{n-k}})\rangle+ \widehat{x}^{(i_1\cdots i_n)}_t- \widehat{x}^{(i_1\cdots i_n)}_s\,,
\end{equation}
for any $i_1,\cdots, i_n\in \{1,\cdots,d\}$. In this way we impose partially an integration by part formula between the coordinates of $x$. By analogy with stochastic calculus, we call the paths $\hat{x}^{(i_1\cdots i_n)} $ associated to $\langle\widehat{\X}_{st},\bullet_{(i_1\cdots i_n)}\rangle$ the \emph{$(i_1\cdots i_n)$-variation} of $x$. The name simple bracket extension comes to distinguish it from the general notion of bracket extension developed in \cite{kelly2012ito}.
\end{rk}
This notion is then sufficient to get a general change of variable formula as explained in \eqref{second_eq_chg} for any choice of parameters $\gamma$ and $d$, which contains also the identity \eqref{first_eq_chg}.
\begin{thm}[Kelly's change of variable formula]\label{chg_thm}
Let $x\colon [0,T]\to \bR^{d}$ be a generic $\gga$-H\"older path  such that there exists a $\gamma$-branched rough path $\mathbf{X}$ over $x$ for any $\gamma\in (0,1)$ and $d$. For any $\gp\in C_b^{N_{\gamma}+1}(\bR^d,\bR)$ and any simple bracket extension $\widehat{\X}$ over $\X$ one has 
\begin{equation}\label{change_of_variable}
\begin{split}
&\gp(x_t)=\\&\gp(x_s)+ \sum_{i=1}^d\int_s^t \partial_i\gP(X_r) d^b\X^i_r+ \sum_{n=2}^{N_{\gamma}} \frac{1}{n!} \sum_{i_1,\cdots, i_n=1}^d\int_s^t\partial_{i_1}\cdots\partial_{i_n}\gP(X_r)d^b\widehat {\X}^{(i_1 \cdots i_n)}_r\,,
\end{split}
\end{equation}
where $\partial_{i_1}\cdots\partial_{i_n}\Phi(X_r)$ is the weighted controlled rough path obtained by composition of $X$ in \eqref{coordinate2} with $\partial_{x_{i_1}}\cdots\partial_{x_{i_n}}\gp\colon \bR^d\to \bR$  as defined in Proposition \eqref{composition}.
\end{thm}


\begin{proof}
Our proof retraces Kelly's original one in \cite{kelly2012ito} via a slightly different notation. We repeat here for sake of completeness. We first remark that for any  choice of  $i_1\,,\cdots\,, i_n\in \{1\,, \cdots\,, d\}$ one has $\partial_{x_{i_1}}\cdots\partial_{x_{i_n}}\gp\in C^{N_{\gga}+1-n}_b(\bR^d, \bR)$ and from Proposition \ref{composition} the weighted controlled rough path $\partial_{i_1}\cdots\partial_{i_n}\gP(X)\in \cD^{(N_{\gamma}+1-n)\gamma}_{\X}$. Moreover we use the explicit structure of $X$ in \eqref{coordinate2} to deduce the explicit form of this weighted controlled rough path
\begin{equation}\label{explicit_formula_branc}
\partial_{i_1}\cdots\partial_{i_n}\gP(X_t)= \sum_{m=0}^{N_{\gamma}- n}\frac{1}{m!}\sum_{j_1,\cdots, j_m =1} ^{d} \partial_{j_1}\cdots\partial_{j_m}\left(\partial_{x_{i_1}}\cdots\partial_{x_{i_n}}\gp(x_t)\right)(\bullet_{j_1}\cdots \bullet_{j_m})\,.
\end{equation}
Thus every rough integral in the right-hand side of \eqref{change_of_variable} is well defined accordingly with Proposition \eqref{rough_integration}, obtaining the identities
\[\begin{split}
&\int_s^t\partial_{i_1}\cdots\partial_{i_n}\gP(X_r)d^b\widehat {\X}^{(i_1 \cdots i_n)}_r=\\& =\sum_{m=0}^{N_{\gamma}-n}\sum_{j_1,\cdots, j_{m}=1} ^{d} \frac{\partial_{j_1}\cdots\partial_{j_m} \left(\partial_{x_{i_1}}\cdots\partial_{x_{i_n}}\gp(x_t)\right)}{m!}\langle \widehat{\X}_{st}, B^+_{(i_1\cdots i_n)}(\bullet_{j_1}\cdots \bullet_{j_{m}}) \rangle+ o(|t-s|)\,,
\end{split}\]
\[\int_s^t \partial_i\gP(X_r) d^b\X^i_r= \sum_{m=0}^{N_{\gamma}-1}\sum_{j_1,\cdots, j_{m}=1} ^{d} \frac{\partial_{j_1}\cdots\partial_{j_m} \left(\partial_{x_{i}}\gp(x_t)\right)}{m!}\langle \widehat{\X}_{st}, B^+_{i}(\bullet_{j_1}\cdots \bullet_{j_{m}}) \rangle+ o(|t-s|)\,,\]
where the second one is a direct consequence of of the hypothesis that $\widehat{\X}$ is over $\X$.
By summing over all indices $i_1,\cdots, i_n$ and $n$, we use the substitution $k=m+n$ to obtain
\begin{equation}\label{multi_expansion}
\begin{split}&\sum_{i=1}^d\int_s^t \partial_i\gP(X_r) d^b\X^i_r+ \sum_{n=2}^{N_{\gamma}} \frac{1}{n!} \sum_{i_1,\cdots, i_n=1}^d\int_s^t\partial_{i_1}\cdots\partial_{i_n}\gP(X_r)d^b\widehat {\X}^{(i_1 \cdots i_n)}_r=\\&= \sum_{k=1}^{N_{\gamma}}\frac{1}{k!}\sum_{i_1,\cdots, i_{k} =1} ^{d} \partial_{i_1}\cdots\partial_{i_k}\gp(x_s)\langle \widehat{\X}_{st},\sum_{n=1}^k\binom{k}{n} B_{(i_1\cdots i_{n})}^+(\bullet_{i_{n+1}}\cdots \bullet_{i_{k}}) \rangle + o(|t-s|)\,.\end{split}
\end{equation}
For any fixed $k$ and $i_1,\cdots ,i_k$, we consider the sum of trees on the right hand side of (\ref{multi_expansion}). Since this expression does not change in its equivalence class by applying any permutation of $k$ by hypothesis on  $\widehat{\X}$ elements we can write 
\begin{equation}\label{multi_expansion2}
\begin{split}
&\langle \widehat{\X}_{st},\sum_{n=1}^k\binom{k}{n} B_{(i_1\cdots i_{n})}^+(\bullet_{i_{n+1}}\cdots \bullet_{i_{k}})\rangle =\\&= \langle \widehat{\X}_{st},\frac{1}{k!} \sum_{\gs\in S_k}\sum_{n=1}^k\binom{k}{n} B_{(i_{\gs(1)}\cdots i_{\gs(n)})}^+(\bullet_{i_{\gs(n+1)}}\cdots \bullet_{i_{\gs(k)}})\rangle\,\\&
= \langle \widehat{\X}_{st},\sum_{\gs\in S_k}\sum_{n=1}^{k-1}\frac{1}{(k-n)!n!} B_{(i_{\gs(1)}\cdots i_{\gs(n)})}^+(\bullet_{i_{\gs(n+1)}}\cdots \bullet_{i_{\gs(k)}})+ \sum_{\gs\in S_k}\frac{1}{k!} \bullet_{(i_{\gs(1)}\cdots i_{\gs(k)})}\rangle\,.\end{split}
\end{equation}
Using the properties $\bullet_{(i_{\gs(1)}\cdots i_{\gs(k)})}\sim \bullet_{(i_{1}\cdots i_{k})} $ and  \[B_{(i_{\gs(1)}\cdots i_{\gs(n)})}^+(\bullet_{i_{\gs(n+1)}}\cdots \bullet_{i_{\gs(k)}})\sim B_{(b_1\cdots b_{n})}^+(\bullet_{a_1} \ldots \bullet_{a_{n-k}})\,,\]
where $a=\{a_1,\cdots,a_{n-k}\}$, $ b=\{b_1,\cdots, b_{n}\}$ is a non empty partition of $\{i_1,\cdots i_k\}$ into two set, it is straightforward to show that the last equation in (\ref{multi_expansion2}) is then equal to 
\begin{equation}\label{key_identity}
\langle \widehat{\X}_{st},\sum_{ \{i_1\,,\cdots, i_k\}=\{a,b\}}B_{(b_1\cdots b_{n})}^+(\bullet_{a_1} \ldots \bullet_{a_{n-k}})+ \bullet_{(i_{1}\cdots i_{k})}\rangle =\langle\X_{st}, \bullet_{i_1} \ldots \bullet_{i_k}\rangle\,,
\end{equation}
by counting the respective multiplicity of the indexes in the two sums and using the main property \eqref{Ibpf} defining a simple bracket extension.
On the other hand, by simply combining the Taylor remainder formula on $\gp$ with the H\"older regularity of $x$ one has
\begin{equation}\label{almost_end_kelly}
\gp(x_t)- \gp(x_s)=\sum_{k=1}^{N_{\gamma}}\frac{1}{k!}\sum_{i_1,\cdots, i_{k} =1} ^{d} \partial_{i_1}\cdots\partial_{i_k}\gp(x_s)\langle\X_{st}, \bullet_{i_1} \ldots \bullet_{i_k}\rangle + o(|t-s|)\,.
\end{equation}
Subtracting formula \eqref{multi_expansion} from \eqref{almost_end_kelly}, we obtain from that the increments of the function
\[A_t=\gp(x_t)- \gp(x_0)+ \sum_{i=1}^d\int_0^t \partial_i\gP(X_r) d^b\X^i_r+ \sum_{n=2}^{N_{\gamma}} \frac{1}{n!} \sum_{i_1,\cdots, i_n=1}^d\int_0^t\partial_{i_1}\cdots\partial_{i_n}\gP(X_r)d^b\widehat {\X}^{(i_1 \cdots i_n)}_r\]
satisfy $ |A_t- A_s|= o(|t-s|)$, thus the function $A$ is constant and equal to $0$. Thereby obtaining the thesis.
\end{proof}
\begin{rk}
In order to have a effective theory related to the notion of simple extension, it is also necessary to prove the existence of at least one simple bracket extension over a branched rough path for any $\gamma\in (0,1)$. This task was completed in \cite[Prop. 5.6]{Friz2020course},  using the standard Lyons-Victoir extension \cite{Lyons2007} in the context of branched rough path, as described in \cite{hairer2015geometric}. However, this result is only partially constructive, meaning that the condition \eqref{Ibpf} expresses the value of a function $\widehat{\X}$ on a set of trees too small to satisfy the definition of branched rough path.
\end{rk}
\subsection{Quasi-geometric bracket extensions}
Theorem \ref{chg_thm} shows us a deterministic change of variable formula which works with every $\gga$-H\"older  path $x\colon [0,T]\to \bR^d$. This formula is extremely general but it relies on a double choice: a particular choice of a $\gamma$-branched rough path $\X$ over $x$ and a simple bracket extension $\widehat{\X}$ over $\X$. Both elements are not unique and checking the relations \eqref{Ibpf}, which are not automatically satisfied, might require an additional effort. In what follows, we will use the notion of quasi-geometric rough paths to provide the notion of \emph{quasi-geometric bracket extension}, which will provide a sufficient condition to build a bracket extension directly from the path $x$. To introduce a quasi-geometric rough path, we define a new alphabet, endowed with a commutative bracket. 
\begin{defn}\label{alphabet_bracket}
For any and $\gga\in (0,1)$ and $d\geq 1$ integer we define the alphabet $\bA_{\gga}^{d}$ as
\[\bA_{\gga}^{d}=\{\ga\in  \bN^d\setminus \{0\}\colon \abs{\ga}\leq N_{\gga}\}\]
where $N_{\gga}$ is defined above and for any $\ga\in \bN^d$, $\ga= (\ga_1,\cdots,\ga_d)$ we set $ \abs{\ga}:= \sum_{i=1}^d \ga_i$. Moreover, we define the function $[,]_{\gamma}\colon \bR^{\bA_{\gga}^{d}}\times\bR^{\bA_{\gga}^{d}}\to\bR^{ \bA_{\gga}^{d}}$ as the unique bilinear map satisfying for any  $\alpha, \beta\in \bA_{\gga}^{d}$ the property
\begin{equation}\label{eq:def_[]N}
[\ga,\gb]_{\gamma}: =\left\{
	\begin{array}{ll}
    \ga +\gb  & \mbox{if } \ga+\gb\in \bA_{\gga}^{d}\,, \\
	0 & \mbox{otherwise}
	\end{array}
	\right.
\end{equation}
where the sum involved is the intrinsic sum between multi-indices.
\end{defn}
Similarly as before, $\{1,\cdots, d\} $ embeds in $\bA_{\gamma}^d$ for any $\gga\in(0,1)$ by simply sending every $i \in\{1\,, \cdots\,, d\}$ to the $i$-th element of the canonical basis. Thereby having   an isomorphism if $1/2<\gga<1$. We can check trivially that $[,]_{\gamma}$ is a commutative bracket. We denote by $\qshuffle_{\gamma}$ the associate quasi shuffle product on $\hat{T}(\bA_{\gamma}^d)$. Furthermore, writing $\bA_{\gamma}^d$ as
\[\bA_{\gamma}^d= \bigcup_{n=1}^{N_{\gamma}}A_n\,,\qquad A_n=\{\ga\in  \bN^d\setminus \{0\}\colon \abs{\ga}=n\}\,,\]
we can write $\bR^{\bA_{\gamma}^d}=\bigoplus_{n=1}^{N_{\gamma}}\bR^{A_n}$ and  we obtain again a decomposition of which is compatible with the commutative bracket $[\,,]_{\gamma}$, as explained in Proposition \ref{prop_natural_grading}. We denote the intrinsic weight by $\Vert\cdot\Vert_{\gamma}$. For any $\ga\in \bA_{\gga}^d$ the maps $(\cI_\alpha)_{\alpha\in \bA_{\gga}^d}\colon \bA_{\gga}^d\to \bA_{\gga}^d$, defined as $ \cI_{\alpha}(w)= w\alpha$ for any $w\in  W(\bA_{\gga}^d)$, they are all integrations map of order $|\alpha|$ with respect to the word length. We introduce the main definition of the section.

\begin{defn}\label{Def_quasi_geom_bracket}
Let $\gga\in (0,1)$ and $d\geq 1$ integer. Every  $\gga$-quasi-geometric rough path $\bX$ associated to $\hat{T}(\bA_{\gamma}^d)$ is said a \emph{quasi-geometric bracket extension}.
\end{defn}
The whole section is devoted to show a constructive procedure to construct a simple bracket extension starting from a  quasi-geometric bracket extension. Naturally, these two objects are defined on two different types of algebraic structure: the first on forests and the second one on words. Moreover, there are also two different alphabets behind. Thus the desired procedure will be based upon a link between these objects. To establish a link between $\cA_{\gamma}^d$ and $\bA_{\gga}^{d}$, we introduce   an explicit function among these two alphabets.
\begin{defn}\label{defn_sym}
Let $\gga\in (0,1)$ and $d\geq 1$ integer. We define the map $S \colon \cA_{\gamma}^d\to \bA_{\gga}^{d}$ on any $a\in \cA_{\gamma}^d$, $a= (i_1\cdots i_n)$,  as
\begin{equation}\label{component}
S(a)= (m_1, \cdots, m_d)\,,
\end{equation}
where for any $j\in \{1,\cdots, d\}$ we set the integers  $m_j:=\sharp \{l\in \{1,\cdots, n\}\colon i_l=j\}$. 
\end{defn}
For instance, if $d=3$ and $\gga\leq 1/3$ one has 
\[S((12))=(1,1,0)\,,\quad S((21))=(1,1,0)\,,\quad S((333))=(0,0,3)\,.\] 
More generally, we can use the commutative bracket to obtain the trivial identity
\begin{equation}\label{alpha_identity}
S((i_1\cdots i_n))= [e_{i_1}\cdots e_{i_n}]_{\gga}\,,
\end{equation}
where $ \{e_j\}_{j\in \{1, \cdots, d\}}$ is the canonical basis of $\bN^d$. By definition, the map $S$ is a well defined surjective application. Moreover, we can also use the application $S$ to describe the equivalence classes of $\cA_{\gga}^d$ with respect to the symmetry relation $\sim$ defined above. 
\begin{lem}\label{combinatorial_lemma}
For any $\gga\in (0,1)$ and $d\geq 1$ integer, $ S\colon (\cA_{\gamma}^d/\sim)\to \bA_{\gga}^{d} $  is bijective.
\end{lem}
\begin{proof}
In order to prove the bijection it is sufficient to show that for every $1\leq n\leq N$, $S$ is a bijection between $A_n$ and $A'_n$ where 
\[A_n= \{1,\cdots, d\}^n/{\sim}\,,\quad A'_n= \{v\in  \bN^d\setminus \{0\}\colon |v|=n\}\,.\]
That is for any $a,b\in \{1,\cdots, d\}^n$, one has $a\sim b$ if and only if $S(a)=S(b)$. Thanks to Definition (\ref{component}), if $a\sim b$ then both $S(a)$ and $S(b)$ belong to $A'_n$ and their components do not change under permutation of the writing of $a$, therefore $S(a)=S(b)$. On the other hand, if $S(a)= S(b)=(m_1, \cdots, m_d)$, then for any $j=\{1,\cdots, d\}$ both $a$ and $b$ contain the same amount of coordinates with the value $j$. Rearranging the components of $a$ and $b$, there exist two permutations $\gs_a$ and $\gs_b$ such that
\begin{equation}\label{canonical}
a_{\gs_{a}}= b_{\gs_b}= (\underset{m_1 \text{ times}}{\underbrace{1\cdots1}}\;\underset{m_2 \text{ times}}{\underbrace{2\cdots 2}}\,\cdots \underset{m_d \text{ times}}{\underbrace{d\cdots d}} )\,.
\end{equation}
By taking the permutation $\gs= \gs_b^{-1}\circ \gs_a$, then $ a_{\gs}=b$ by construction.
\end{proof}
A direct consequence of this result is the following elementary combinatorial identity. If $c_{i_1\cdots i_n}$,  is a sequence of values indexed by $ \{1,\cdots, d\}^n $, $n\geq 1$, which is invariant by permutation, then  we have the identity
\begin{equation}\label{identity_sum}
\sum_{i_1, \cdots, i_n=1}^d c_{i_1\cdots i_n}= \sum_{\gb\in \bN^d \colon \abs{\gb}=n}\frac{n!}{\gb_1!\cdots \gb_d!}c_{\gb'}\,,
\end{equation}
where $\gb'\in \{1,\cdots, d\}^n$  is defined as
\[\gb'=(\underset{\gb_1 \text{ times}}{\underbrace{1\cdots1}}\;\underset{\gb_2 \text{ times}}{\underbrace{2\cdots 2}}\,\cdots \underset{\gb_d \text{ times}}{\underbrace{d\cdots d}} )\,.\]
This simple identity allows to justify the equivalence between the definitions \eqref{composition_ope} and \eqref{composition_ope2}. Even if the map $S$ is originally defined at the level of the alphabets, we can easily extend it at the level of forests.
\begin{prop}
Let $\gga\in (0,1)$ and $d\geq 1$ integer. By simply applying the application $S$ in Definition \ref{defn_sym} at each decoration on any element of $\cF(\cA_{\gamma}^d)$ and extending it linearly to preserve $\1$, one has a Hopf Algebra morphism $ \cS \colon \cH(\cA_{\gamma}^d)\to \cH(\bA_{\gga}^{d})$. We call this application the \emph{symmetrization map}.
\end{prop}
\begin{proof}
Since $S$ acts only at the level of decoration, it is straightforward to verify the conditions
\[
\cS (\t_1\t_2)= \cS(\t_1) \cS(\t_2)\,,\quad (\cS\otimes \cS)\gD\t_1= \gD \cS(\t_1)\,,
\]
for any $\t_1$, $\t_2\in \cH(\cA_{\gamma}^d)$. Thereby obtaining the thesis.
\end{proof}
The symmetrisation map represents the first operation of our connection. To conclude the procedure, we introduce and  explicit linear map, which connects trees with forests.
\begin{prop}
For any  alphabet $A$ endowed with and a commutative bracket, there exists a unique Hopf algebra morphism $\psi\colon \cH(A)\to \hat{T}(A)$, defined recursively on $\cF(A)$ by the following conditions: 
\begin{itemize}
\item  for any $a\in A$ and any forest $f\in \cF(A)$
\begin{equation}\label{quasi_arborification1}
\psi( B_{a}^+(f)):=(\psi(f)) a\,;
\end{equation}
\item $\psi(\1):=\1 $ and for every couple of forests $\gs,\t\in \cF(A)$
\begin{equation}\label{quasi_arborification2}
\psi(\gs\t):= \psi(\gs)\qshuffle \psi(\t)\,.
\end{equation}
\end{itemize}
We call this application the \emph{contracting arborification}.
\end{prop}
\begin{proof}
Conditions \eqref{quasi_arborification1} and \eqref{quasi_arborification2} identify uniquely an algebra morphism between $\cH(A)$ and $\hat{T}(A)$. Thus the proposition will follow, once we prove for any $ h \in \cH(A)$ the identity 
\begin{equation}\label{hopf_morphism}
\nabla\psi (h)= (\psi\otimes\psi)\gD h\,. 
\end{equation}
To prove this identity, we will follow essentially the same argument as \cite[Lem. 4.8]{hairer2015geometric} but in the  context of quasi-shuffle product. We repeat it here for the sake of completeness. Since both sides of \eqref{hopf_morphism} are linear  in $h$, it is then sufficient to prove \eqref{hopf_morphism} on any forest $f\in\cF(A)$ working by induction on the forest cardinality $|\cdot|$. The basis of induction is trivial by definition. Supposing the identity  \eqref{hopf_morphism} true on any forests $h'$ such that $\abs{h'}\leq m$, we consider a forest  $h$ such that $\abs{h}= m+1$. If this  forest is under form $h=B^{+}_a(f)$ for some $f\in \cF(A)$, by definition of $\psi$ in \eqref{quasi_arborification1} and the coproduct $\nabla$ one has
\begin{equation}\label{hopf_morphism1}
\nabla \psi(B^{+}_a(f))=\psi(f)a\otimes \1 +\nabla(\psi(f)) (\1 \otimes a)\,. 
\end{equation}
By induction one has $\nabla\psi(f)= (\psi\otimes\psi)\gD f$ and the right hand side of \eqref{hopf_morphism1} becomes
\[
(\psi\otimes\psi)(B^{+}_a(f)\otimes \1+ (id\otimes B^{+}_a )\gD f )= (\psi\otimes\psi)\gD B^{+}_a(f)\,.
\]
On the other hand, supposing that  $h=h_1h_2$ for some forests $h_1$, $h_2$ satisfying $ \abs{h_1} \vee\abs{h_2}\leq m$, we use  the bialgebra properties of $\nabla$ and $\qshuffle$ to obtain
\begin{equation}\label{hopf_morphism2}
\nabla \psi(h_1h_2)= \nabla (\psi(h_1)\qshuffle \psi(h_2))= (\nabla \psi(h_1))\qshuffle(\nabla\phi(h_2))\,,
\end{equation}
where we denote again by $\qshuffle$ the tensorisation of the quasi shuffle product on $\hat{T}(A)\otimes\hat{T}(A)$. The recursive hypothesis tells us that $\nabla (\psi(h_i))= (\psi\otimes\psi)\gD h_i$ for $i=1,2$, therefore the definition of $\psi$ as well as the bialgebra property of $\gD$ allows us to write the right hand side of (\ref{hopf_morphism2}) as
\[(\psi\otimes\psi)\gD h_1 \qshuffle (\psi\otimes\psi)\gD h_2= (\psi\otimes\psi)\gD h_1 \cdot \gD h_2 = (\psi\otimes\psi)\gD h_1h_2\,.\]
Thereby closing the induction.
\end{proof}

\begin{rk}
The map $\psi$ has already been introduced in \cite{curry2020} and it is a standard tool in the algebraic literature. Its name, contracting arborification comes as a tribute to Jean Lecalle's arborification apparatus in the context of Hopf algebras of trees, as explained in \cite{fauvet17}. We recall that in the shuffle case  a  function with the same properties was introduced in  \cite{hairer2015geometric} and it was denoted by $\phi_g$. Generally speaking, this proposition is a direct consequence of the Universal property for $\cH(A)$, see \cite[Theorem 2]{Connes1998}.
\end{rk}
By composing the symmetrisation map and the contracting arborification with the alphabet $\bA_{\gamma}^d $, we obtain an explicit Hopf algebra morphism $\psi_{\cS}:=(\psi\circ \cS)\colon \cH(\cA_{\gamma}^d)\to \hat{T}(\bA_{\gga}^d)$. Considering the adjoint map $\psi_{\cS}^*= \cS^* \circ \psi^*\colon \hat{T}(\bA_{\gga}^d)^{*}\to\cH(\cA_{\gamma}^d)^* $, we  obtain  the right connection.
\begin{thm}\label{big_thm}
Let $\gga\in (0,1)$ and $d\geq 1$ integer. For any quasi-geometric bracket extension $\bX$ the function $\psi_{\cS}^*\bX$ is a simple bracket extension.
\end{thm}
\begin{proof}
The theorem will follow by showing that $\psi_{\cS}^*\bX$  satisfies the properties in  Definition  \ref{bracket}. This task is obtained by simply analysing some additional properties of $\psi_{\cS}$. Indeed it is sufficient to check the following two properties: firstly, for any $f\in \cF(\cA_{\gamma}^d)$ one has
\begin{equation}\label{algebraic_identity1}
\psi_{\cS}(f)\in \langle w\in W(\bA_{\gamma}^d) \colon \Vert w \Vert_{\gamma}= |f|_{\gga} \rangle\,.
\end{equation}
Secondly, for all $1 \leq n \leq N_{\gamma}$ and all $i_1, \cdots, i_n \in \{1,\cdots, d\}$ we have the identity
\begin{equation}\label{algebraic_identity2}
\psi_{\cS} (\llangle i_1 \cdots i_n \rrangle) =[e_{i_1}\cdots e_{i_n}]_{\gga}\,.
\end{equation}
Let us show firstly that from \eqref{algebraic_identity1} and \eqref{algebraic_identity2} $ (\psi_{\cS})^*\bX$ is a bracket extension. As a consequence of \eqref{algebraic_identity1}  we will obtain that $(\psi_{\cS})^*\bX$ is a well-defined $\gga$-weighted rough path over $\cH(\cA_{\gamma}^d)$ with respect to its intrinsic weight. Moreover, it comes immediately from the definition of $\cS$ that for  any couple of forests $\tau, \gs$ such that $\tau\sim \gs$ we have $\cS(\tau)= \cS(\gs)$ and consequently
\[\langle\psi_{\cS}^*\bX_{st},\gs \rangle=\langle\psi^*\bX_{st},\cS\gs \rangle=\langle \psi^*\bX_{st},\cS\tau \rangle= \langle\psi_{\cS}^*\bX_{st},\tau \rangle\,,\]
for any $s,t\in [0,T]$. Finally, we conclude from \eqref{algebraic_identity2} and the  \eqref{alpha_identity} that one has the trivial identity $\psi_{\cS}(\bullet_{(i_1\cdots i_n)})=[e_{i_1}\cdots e_{i_n}]_{\gga}$ and the final property for any $s,t\in [0,T]$
\[\langle\psi_{\cS}^*\bX_{st}, \llangle i_1 \cdots i_n \rrangle\rangle = \langle \bX_{st}, \psi_{\cS}(\llangle i_1 \cdots i_n \rrangle)\rangle = \langle\bX_{st}, [e_{i_1}\cdots e_{i_n}]_{\gga} \rangle= \langle\psi_{\cS}^*\bX_{st}, \bullet_{(i_1\cdots i_n)}\rangle\,. \]
To conclude the proof we  will show the properties \eqref{algebraic_identity1} and \eqref{algebraic_identity2} by induction. In case of \eqref{algebraic_identity1}, we will prove it  by induction on $|\cdot|_{\gamma}$, where  the case $f=\1 $ is trivial. Supposing the identity \eqref{algebraic_identity1} true for any $f'\in \cF(\cA_{\gamma}^d)$ such that $|f|_{\gamma}\leq m $ for some integer $m$, then we will prove that \eqref{algebraic_identity1} holds for any $f\in \cF(\cA_{\gamma}^d)$ such that $|f|_{\gamma}=m+1$. If $f= B^+_{(i_1\cdots i_n)}(\gs)$ the induction hypothesis implies that $\psi_{\cS}(\gs)$ is a linear combination of words $w'$ such that $\Vert w'\Vert_{\gamma}= |\gs|_{\gamma}$. By definition of $\psi$ in \eqref{quasi_arborification1} and accordingly to the maps $S$ in Definition \ref{defn_sym}, $\psi_{\cS}(f)$ will be a linear combination of the words $w'S((i_1\cdots i_n))$. Since one has the trivial identity $\Vert S((i_1\cdots i_n))\Vert_{\gamma}=n$ and $|\gs|_{\gga}+n=m+1$ we conclude in this first case. Secondly if $f= f_1f_2$ for some non empty forest. The homomorphism property of $\psi_{\cS}$ and the properties of the intrinsic weight $\Vert \cdot \Vert_{\gamma}$ with the quasi-shuffle product imply that $\psi_{\cS}(f)$ is a linear combination of words $w''$ satisfying $\Vert w''\Vert= |f_1|_{\gamma}+|f_2|_{\gamma}=|f|=m+1$. Therefore we obtain the property  \eqref{quasi_arborification1}. The identity \eqref{quasi_arborification1} will be proved by induction on the index $n$. If $n=1$ the identity \eqref{algebraic_identity2} is verified by definition. Supposing  that \eqref{algebraic_identity2} holds for some $1\leq n <N_{\gamma}$, we prove that
\[ \psi_{\cS} (\llangle i_1 \cdots i_{n+1} \rrangle)=[e_{i_1}\cdots e_{i_{n+1}}]_{\gga}\,.\]
Starting from formula \eqref{defn_alternative_bracket}, we have the following identities
\[
\begin{split}
\llangle i_1 \cdots i_{n+1} \rrangle &= \bullet_{i_1} \ldots \bullet_{i_{n+1}}- \mathbf{B}^+\big(\gD'(\bullet_{i_1}\cdots \bullet_{i_{n+1}})\big)\\&=\bullet_{i_1} \ldots \bullet_{i_{n+1}}- \mathbf{B}^+\big(\bullet_{i_{n+1}}\otimes \bullet_{i_1}\cdots \bullet_{i_{n}}+ \bullet_{i_1}\cdots \bullet_{i_{n}} \otimes \bullet_{i_{n+1}}  \\&+(\bullet_{i_{n+1}}\otimes \1)\gD'(\bullet_{i_1}\cdots \bullet_{i_{n}})+ (\1\otimes\bullet_{i_{n+1}} )\gD'(\bullet_{i_1}\cdots \bullet_{i_{n}})\big)\\&= \bullet_{i_1}\cdots \bullet_{i_{n+1}}- B_{(i_1\cdots i_n)}^+(\bullet_{i_{n+1}})-   B_{i_{n+1}}^+(\bullet_{i_1}\cdots \bullet_{i_n})\\& -\mathbf{B}^+((\bullet_{i_{n+1}}\otimes \1)\gD'(\bullet_{i_1}\cdots \bullet_{i_{n}}))-\mathbf{B}^+((\1\otimes\bullet_{i_{n+1}} )\gD'(\bullet_{i_1}\cdots \bullet_{i_{n}}))\,.
\end{split}
\]
Applying $\psi_{\cS}$ to both sides, we obtain
\begin{equation}\label{recursive_polynomial2}
\begin{split}
\psi_{\cS}(\llangle i_1 \cdots i_{n+1}\rrangle)&=\psi_{\cS}(\bullet_{i_1}\cdots \bullet_{i_n})\qshuffle_{\gamma} \;e_{i_{n+1}} -  e_{i_{n+1}}[e_{i_1}\cdots e_{i_{n}}]_{\gga}\\&-\psi_{\cS}(\bullet_{i_1}\cdots \bullet_{i_n})e_{i_{n+1}}- \psi_{\cS}(\mathbf{B}^+(\bullet_{i_{n+1}}\otimes \1)\gD'(\bullet_{i_1}\cdots \bullet_{i_n}))\\&- \psi_{\cS}(\mathbf{B}^+(\1\otimes \bullet_{i_{n+1}})\gD'(\bullet_{i_1}\cdots \bullet_{i_n}))\,.
\end{split}
\end{equation}
Writing the induction hypothesis as
\[\psi_{\cS}(\bullet_{i_1}\cdots \bullet_{i_n})=\psi_{\cS}( \mathbf{B}^+\gD'(\bullet_{i_1}\cdots \bullet_{i_n}) )+ [e_{i_1}\cdots e_{i_{n}}]_{\gga}\,,\]
we can rewrite the equation \eqref{recursive_polynomial2} as
\[\begin{split}
\psi_{\cS}(\llangle i_1 \cdots i_{n+1}\rrangle)&=\psi_{\cS}( \mathbf{B}^+\gD'(\bullet_{i_1}\cdots \bullet_{i_n}) )\qshuffle_{\gamma}\; e_{i_{n+1}}+[e_{i_1}\cdots e_{i_{n}}]_{\gga}\qshuffle_{\gamma} \;e_{i_{n+1}}\\&-  e_{i_{n+1}}[e_{i_1}\cdots e_{i_{n}}]_{\gga}  -\psi_{\cS}( \mathbf{B}^+\gD'(\bullet_{i_1}\cdots \bullet_{i_n}) )e_{i_{n+1}}-[e_{i_1}\cdots e_{i_{n}}]_{\gga}e_{i_{n+1}}  \\&- \psi_{\cS}( \mathbf{B}^+(\bullet_{i_{n+1}}\otimes \1)\gD'(\bullet_{i_1}\cdots \bullet_{i_n}))- \psi_{\cS}( \mathbf{B}^+(\1\otimes \bullet_{i_{n+1}})\gD'(\bullet_{i_1}\cdots \bullet_{i_n}))\,.
\end{split}\]
After a simplification of the words with two letters, the right-hand side becomes
\begin{equation}\label{recursive_polynomial3}
\begin{split}
&[e_{i_1}\cdots e_{i_{n+1}}]_{\gga} +\psi_{\cS}( \mathbf{B}^+\gD'(\bullet_{i_1}\cdots \bullet_{i_n}) )\qshuffle_{\gamma}\;e_{i_{n+1}}- \psi_{\cS}( \mathbf{B}^+\gD'(\bullet_{i_1}\cdots \bullet_{i_n}))e_{i_{n+1}}\\& -\psi_{\cS}( \mathbf{B}^+(\1\otimes \bullet_{i_{n+1}})\gD'(\bullet_{i_1}\cdots \bullet_{i_{n}}))- \psi_{\cS}( \mathbf{B}^+(\bullet_{i_{n+1}}\otimes \1)\gD'(\bullet_{i_1}\cdots \bullet_{i_{n}}))\,.
\end{split}
\end{equation}
Let us show that the remaining terms after $[e_{i_1}\cdots e_{i_{n+1}}]_{\gga}$ in \eqref{recursive_polynomial3} are zero. Using again the notation in the identity \eqref{key_combinatorial} and the definition of $\mathbf{B^{+}}$, we write the quantities
\[\begin{split}
I=\mathbf{B}^+\gD'(\bullet_{i_1}\cdots \bullet_{i_n})&= \sum_{\{a,b\}=\{i_1,\cdots,i_n\}}B^{(b_1 \cdots b_{n-k})}_+(\bullet_{a_1} \cdots \bullet_{a_k})\,,\\
II=\mathbf{B}^+(\bullet_{i_{n+1}}\otimes \1)\gD'(\bullet_{i_1}\cdots \bullet_{i_n})&= \sum_{\{a,b\}=\{i_1,\cdots,i_n\}}B^{(b_1 \cdots b_{n-k})}_+(\bullet_{a_1} \cdots \bullet_{a_k}\bullet_{i_{n+1}})\,,\\
III=\mathbf{B}^+(\1\otimes \bullet_{i_{n+1}})\gD'(\bullet_{i_1}\cdots \bullet_{i_n})&=\sum_{\{a,b\}=\{i_1,\cdots,i_n\}}B^{(b_1 \cdots b_{n-k}i_{n+1})}_+(\bullet_{a_1} \cdots \bullet_{a_k})\,.\end{split}\]
Applying finally the properties of $\qshuffle_{\gga}$ and the identity $e_{i_{n+1}}= \psi_{\cS}(\bullet_{i_{n+1}})$, we have
\[\begin{split}
\psi_{\cS}(\mathbf{B}^+\gD'(\bullet_{i_1}\cdots \bullet_{i_n}))\qshuffle_{\gga}\;e_{i_{n+1}} &=\sum_{\{a,b\}=\{i_1,\cdots,i_n\}}(\psi_{\cS}(\bullet_{a_1} \cdots \bullet_{a_k})[e_{b_1} \cdots e_{b_{n-k}}]_{\gga})\qshuffle e_{i_{n+1}}\\& =\sum_{\{a,b\}=\{i_1,\cdots,i_n\}}\psi_{\cS}(\bullet_{a_1} \cdots \bullet_{a_k})[e_{b_1} \cdots e_{b_{n-k}}]_{\gga}e_{i_{n+1}}\\&+\sum_{\{a,b\}=\{i_1,\cdots,i_n\}} (\psi_{\cS}(\bullet_{a_1} \cdots \bullet_{a_k} \bullet_{i_{n+1}}))[e_{b_1} \cdots e_{b_{n-k}}]_{\gga}\\&+ \sum_{\{a,b\}=\{i_1,\cdots,i_n\}}\psi_{\cS}(\bullet_{a_1} \cdots \bullet_{a_k})[e_{b_1} \cdots e_{b_{n-k}}e_{i_{n+1}}]_{\gga}\\&=\psi_{\cS}(I)e_{i_{n+1}}+\psi_{\cS}(II)+ \psi_{\cS}(III) \,.\end{split}\]
Thereby obtaining the thesis.
\end{proof}
\begin{rk}
One of the main advantage using quasi-geometric rough paths to build a bracket extensions lies in Theorem \ref{fund_thm_quasi}, where it is possible to transform geometric structure into quasi-geometric. Thus, we can in principle build some example of bracket extensions starting from a geometric rough path defined over a tensor algebra on a vector space containing $\bR^d$. Therefore it is possible to apply the wide literature of  geometric rough paths in order to construct non-trivial examples of simple bracket extensions.
\end{rk}
We conclude the section by showing the relation between the simple bracket extension in case $\gga\in (1/3,1/2)$ and the notion of Ito and Stratonovich extension of an elementary $\gga$-rough path given in  Proposition \ref{def_rough_Ito_and_strato}.  As a first observation, we simply remark that the alphabet $\bA^d_2$ in  \eqref{alphabet_ito_extension} and its commutative bracket coincide exactly with Definition \ref{alphabet_bracket} when $ \gga\in (1/3,1/2)$.  Thus one has that $X^I$, the It\^o extension of an elementary $\gga$-rough path $X=(x, \bX)$ is an example of a quasi-geometric bracket extension, accordingly to Proposition \ref{def_rough_Ito_and_strato} and Definition \ref{Def_quasi_geom_bracket}. Moreover, when  $ \gga\in (1/3,1/2)$ it follows easily from Definition \ref{defn_branched} that  there is a trivial one-to-one correspondence between $\gga$-branched rough paths $\X$  and elementary $\gga$-rough paths  $X=(x, \bX)$, which is given by the following identities
\begin{equation}\label{eq_rel_bra_elem}
x^{i}_t- x^i_s= \langle \X_{st},\bullet_i \rangle\,, \quad \bX_{st}^{ji}= \langle \X_{st},B^{+}_i(\bullet_j) \rangle\,.
\end{equation}
Putting together these two facts with Theorem \ref{big_thm}, we obtain a full characterisation of the simple bracket extension in the special case $ \gga\in (1/3,1/2)$.
\begin{prop}\label{prop_trivial_bracket}
Let $\gga\in (1/3,1/2)$ and $d\geq 1$ integer. For every  $\gga$-branched rough path $\X$ associated to an elementary rough path $X= (x, \bX)$ via the relations \eqref{eq_rel_bra_elem}, there exists a unique simple bracket extension  $\widehat{\X}$ over $\X$ which is given by $\psi_{\cS}^*X^I$, where $X^I$ is the It\^o extension of $X$, as described in Proposition \ref{prop_ito_and_strato_ext}.
\end{prop}
\begin{proof}
Combining Proposition \ref{def_rough_Ito_and_strato} and  Theorem \ref{big_thm} with the relations \eqref{eq_rel_bra_elem}, we obtain that $\psi_{\cS}^*X^I$ is a simple bracket extension over $\X$. On the other hand, starting from a generic simple bracket extension $\widehat{\X}$ over $\X$, the hypothesis $\gga\in (1/3,1/2)$ implies that $\widehat{\X}$ is branched rough path defined in the dual of the space of the following vector space
\[\cH^2_{|\cdot|_{\gga}}(\cA_{\gga}^d)= \langle \1, \bullet_i,  B^{+}_i(\bullet_j), \bullet_j \bullet_i,\bullet_{(ij)}\rangle\,,\quad  i,j\in\{1\,,\cdots \,, d \}\,. \]
In order to identify $\widehat{\X}$, it is sufficient to show that  $\langle\widehat{\X}_{st},\bullet_{(ij)}\rangle= \langle\psi_{\cS}^*X^I_{st},\bullet_{(ij)}\rangle$. Using the identity \eqref{Ibpf}  with the relations \eqref{eq_rel_bra_elem} and  \eqref{def_rough_Ito_and_strato}, for any $i,j\in\{1\,,\cdots \,, d \}$  one has
\begin{equation}\label{relation_12}
\begin{split}
\langle\widehat{\X}_{st},\bullet_{(ij)}\rangle&= \langle\X_{st}, \bullet_i \bullet_j - B^{+}_i(\bullet_j),- B^{+}_j(\bullet_i)\rangle\\&=(x^i_t-x^i_s) (x^j_{t}-x^j_s) - \bX^{ij}_{st}- \bX^{ji}_{st}\\&= \langle X^I_{st} , [e_ie_j]\rangle= \langle \psi_{\cS}^* X^I_{st} ,,\bullet_{(ij)}\rangle\,.
\end{split}
\end{equation}
Thereby completing the proof.
\end{proof}
\subsection{Quasi-geometric change of variable formula}
Recalling from Theorem \ref{chg_thm} that any choice of a  simple bracket extension implies a change of variable formula, another important consequence of Theorem \ref{big_thm} is the existence of a quasi-geometric  change of variable formula  over a $\gga$-H\"older path $x\colon [0,T]\to \bR^{d}$, $x=(x^1,\cdots, x^d)$. We conclude the section by stating and proving this formula via two different methods.

In this case, we suppose that there exists a quasi-geometric bracket extension $\bX$ over $x$, meaning that for any $i\in\{1\,, \cdots,d\} $ one has $\langle \bX_{st}, e_i\rangle= x^i_t- x^i_s$. Similarly as before, the  path  $ X\colon [0,T]\to (T(\bA_{\gga}^d))^{d}$, $X=\{X^{i}\colon [0,T]\to T(\bA_{\gga}^d)\}_{i\in\{1\,, \cdots,d\}}$ defined for any word $w\in W(\bA_{\gga}^d)$ as
\begin{equation}\label{coordinate3}
\langle w,(X_t)^i\rangle= \begin{cases} x_t^i & \text{if}\;w=\1^*\,,\\
1 & \text{if}\; w=e_i \\ 0 & \text{otherwise}\,,
\end{cases}
\end{equation}
satisfies immediately the properties of Definition \ref{defn_contr_geo}, obtaining $X\in (\cD^{N\gamma}_{\bX})^d$ for any $N\geq 1$. The definition of $X$ implies an explicit algebraic identity.
\begin{lem}\label{composition_quasi_geometric}
Let $\gga\in (0,1)$ and $d\geq 1$ integer. For any integer $1\leq N\leq N_{\gga}+1 $ and any function $\gp\in C^{N}_b(\bR^{d}, \bR)$ the weighted modelled distribution $ \gP(X_t)\in \cD^{N\gamma}_{\bX}$ is given by the following identity
\begin{equation}\label{explicit_formula_2}
\gP(X_t)= \sum_{\substack{u\in W(\bA_{\gga}^{d})\\ \Vert u\Vert_{\gamma}< N}}\frac{\partial^{u}\gp(x_t)}{u!}u\,,
\end{equation}
where for any word  $u\in W(\bA_{\gga}^{d})$ we define 
\[\partial^u\gp(x_t):= \begin{cases} \gp(X_t) & \text{if}\;u=\1\,,\\
\partial^{k_1+\cdots +k_n}\gp(x_t) & \text{if} \;u=k_{1}\cdots\,k_{n} ,
\end{cases}\; u!:=  \begin{cases} 1& \text{if}\;u=\1\,,\\
\prod_{i=1}^n k_i! & \text{if}\; u=k_{1}\cdots\,k_{n} \,.
\end{cases}\]
\end{lem}
\begin{proof}
Following the definition of $\gP$ in \eqref{composition_ope2} and the definition of $X$ in \eqref{coordinate3}, we obtain
\[\gP(X_t)=\sum_{0\leq |k|< N}\frac{\partial^k\gp(x_t)}{k_1!\cdots k_{d}!}e_1^{\qshuffle_{\gga} k_1}\qshuffle_{\gga}\cdots \qshuffle_{\gga} e_d^{\qshuffle_{\gga} k_{d}}\,,\]
where  $e_1\,,\cdots\,, e_d\in \bA_{\gga}^{d}$ is the canonical basis and $k=(k_1,\cdots, k_d)$ is generic multi index. Thus we obtain the thesis  as long as we have
\begin{equation}\label{comb_lemma}
\sum_{0\leq |k|< N}\frac{\partial^k\gp(x_t)}{k_1!\cdots k_{d}!}e_1^{\qshuffle_{\gga} k_1}\qshuffle_{\gga}\cdots \qshuffle_{\gga} e_d^{\qshuffle_{\gga} k_{d}}\,= \sum_{\substack{u\in W(\bA_{\gga}^{d})\\ \Vert u\Vert_{\gamma}< N}}\frac{\partial^{u}\gp(x_t)}{u!}u\,,
\end{equation}
We check this identity by induction on $d$. In case $d=1$ the underlying alphabet becomes  $\bA_{\gga}^{1}=\{1\,,\cdots \,, N_{\gga}\}$. Recalling the standard identity using the shuffle product $\shuffle$ 
\[a^{\shuffle k}=k!\;\underset{k \text{ times}}{\underbrace{a\cdots a}}\,,\]
for any letter $a$ and any alphabet $A$, we use the maps $\exp$  and $\log$ defined in Theorem \ref{hoffmann_iso} to obtain for  any integer $k\geq 0$ the following identity
\begin{equation}\label{identity_quasi_shuffle}
\begin{split}
1^{\qshuffle_{\gga} k}&=\exp (\log(1^{\qshuffle_{\gga} k}))=\exp ((\log(1))^{\shuffle k})\\&= k!\exp (\underset{k \text{ times}}{\underbrace{1\cdots 1}})= \sum_{I\in C(k)}\frac{k!}{I!}[\underset{k \text{ times}}{\underbrace{1\cdots 1}}]_{\gga,I}\,.
\end{split}
\end{equation}
Summing this identity over $k$ we obtain the base of the induction
\[\sum_{k=0}^{N-1}\frac{\partial^k\gp(x_t)}{k!}1^{\qshuffle_{\gga} k}=\sum_{k=0}^{N-1}\sum_{I\in C(k)}\frac{\partial^k\gp(x_t)}{I!}[\underset{k \text{ times}}{\underbrace{1\cdots 1}}]_{\gga,I} =\sum_{\substack{u\in W(\bA_{\gga}^{d})\\ \Vert u\Vert_{\gamma}< N}}\frac{\partial^{u}\gp(x_t)}{u!}u\,,\]
because the couples $(k, I)$ are in  one-to-one correspondence with the words $\Vert u\Vert_{\gamma}< N$, thanks to the hypothesis $1\leq N\leq N_{\gga}+1 $. Supposing the identity \eqref{comb_lemma} true for $d-1$, we write the left-hand side of \eqref{comb_lemma} as
\[\sum_{n=0}^{N-1} \frac{1}{n!}\bigg(\sum_{\substack{k\in \bN^{d-1} \\|k|< N-n}}\frac{\partial^{k}(\partial^n_{d}\gp)(x_t)}{{k_1!\cdots k_{d-1}!}}e_1^{\qshuffle_{\gga} k_1}\qshuffle_{\gga}\cdots \qshuffle_{\gga} e_{d-1}^{\qshuffle_{\gga} k_{d-1}}\bigg)\qshuffle_{\gga}  e_d^{\qshuffle_{\gga} k_{d}}\,,\]
Applying the induction hypothesis and the identity \eqref{identity_quasi_shuffle} on the last term in the external product, this sum becomes 
\begin{equation}\label{exact_induction2}
\sum_{n=0}^{N-1} \sum_{I\in C(n)}\sum_{\substack{u'\in W((\bA_{\gga}^{d-1})')\\ \Vert u'\Vert_{\gamma}< N-n}}\frac{\partial^{u'}(\partial^n_{d}\gp)(x_t)}{u'!I!}u'\qshuffle_{\gga} [\underset{n \text{ times}}{\underbrace{e_d\cdots e_d}}]_{\gga,I}\,,
\end{equation}
where $(\bA_{\gga}^{d-1})'\subset \bA_{\gga}^{d}$ is the alphabet obtained by adding a zero component to each term of $\bA_{\gga}^{d-1}$. Using the direct definition of the quasi-shuffle product given in \eqref{defn_alt_quasi_shuffle}, the sum \eqref{exact_induction2} becomes
\[\sum_{n=0}^{N-1} \sum_{I\in C(n)}\sum_{\substack{u'\in W((\bA_{\gga}^{d-1})')\\ \Vert u'\Vert_{\gamma}< N-n}}\frac{\partial^{u'}\partial^n_{d}\gp(x_t)}{u'!I!}\sum_{f\in\bS_{\abs{u'},|I|}} f(u'[\underset{n \text{ times}}{\underbrace{e_d\cdots e_d}}]_{\gga, I})\,,\]
where $\bS_{\abs{u'},|I|}$ is the set of all surjections $f\colon \{1\,, \cdots \,, \abs{u'}+|I|\}\to \{1\,, \cdots \,, k\}$  satisfying $f(1)<\cdots < f(\abs{u'}), f(\abs{u'} + 1) <\cdots < f(\abs{u'}+|I|)$ for some integer $k$ such that $\max(\abs{u'},|I|)\leq k\leq \abs{u'}+|I|$. Since all the letters in  $u'$ have the $d$-th component equal to zero, the elements $(n,I, u', f)$ above are in  one-to-one correspondence with all words $u\in W(\bA_{\gga}^{d})$ such that $\Vert u\Vert_{\gga}<N$.  Moreover it follows from their definitions that  we have $\partial^{u'}\partial^n_{d}\gp= \partial^{u}\gp$ and $u! = u'!I!$. Thus we complete the induction.
\end{proof}
From the explicit formula of $\Phi(X_t)$, we  deduce  the change of variable formula.
\begin{thm}\label{quasi_shuffle_change_thm}
Let $x\colon [0,T]\to \bR^{d}$ be a generic $\gga$-H\"older path for any  $\gga\in (0,1)$ and $d\geq 1$ integer such that there exists a $\gamma$-quasi-geometric rough path $\bX$ over $x$. For any $\gp\in C_b^{N_{\gamma}+1}(\bR^d,\bR)$ one has 
\begin{equation}\label{quasi_shuffle_change}
\gp(x_t)=  \gp(x_s) +\sum_{n=1}^{N_{\gga}}\sum_{\substack{ |k|=n}}\frac{1}{k!} \int_s^t\partial^{k}\gP(X_r)d^g \bX^{k}_r\,,
\end{equation}
where $k$ is generic multi-index and $\partial^{k}\gP(X_r)$ is the lifting of  is the weighted controlled rough path obtained by composition of $X$ in \eqref{coordinate3} with $\partial^{k}\gp\colon \bR^d\to \bR$  as defined in Proposition \eqref{composition}.
\end{thm}
\begin{proof}
For any $k\in \bA_{\gga}^{d}$ such that $|k|=n$ we can apply  the explicit formula \eqref{explicit_formula_2} to $\partial^{k}\gP(X_r)$ obtaining
\[\partial^{k}\gP(X_r)=\sum_{\substack{u\in W(\bA_{\gga}^{d})\\ \Vert u\Vert_{\gamma}< N_{\gga}+1-n}}\frac{\partial^{u}(\partial^{k}\gp)(x_t)}{u!}u\,.\]
Applying Proposition \eqref{rough_integration} and  \eqref{equ:rough_int_est} to the geometric rough integral we deduce also
\begin{equation}\label{equations_quasi}
\frac{1}{k!}\int_s^t\partial^{k}\gP(X_r)d \bX^{k}_r= \sum_{\substack{u\in W(\bA_{\gga}^{d})\\ \Vert u\Vert_{\gamma}< N_{\gga}+1-n}}\frac{\partial^{u}(\partial^{k}\gp)(x_t)}{u!k!}\langle \bX_{st}, u k\rangle + o(\abs{t-s})\,.
\end{equation}
Summing both sides of \eqref{equations_quasi} over $n$ and all  $k$, we obtain the trivial identity
\begin{equation}\label{equations_quasi2}
\begin{split}
&\sum_{n=1}^{N_{\gga}}\sum_{\substack{k \in \bA_{\gga}^d\\ |k|=n}}\frac{1}{k!}\int_s^t\partial^{k}\gP(X_r)d \bX^{k}_r= \\&= \sum_{n=1}^{N_{\gga}}\sum_{\substack{k \in \bA_{\gga}^d\\ |k|=n}}\sum_{\substack{u\in W(\bA_{\gga}^{d})\\ \Vert u\Vert_{\gamma}< N_{\gga}+1-n}}\frac{\partial^{u}(\partial^{k}\gp)(x_t)}{u!k!}\langle \bX_{st}, u k\rangle + o(\abs{t-s})\\&=\sum_{\substack{w\in W(\bA_{\gga}^{d})\\ 1\leq \Vert w\Vert_{\gamma}< N_{\gga}+1}}\frac{\partial^{w}\gp(x_t)}{w!}\langle \bX_{st}, w\rangle+ o(\abs{t-s})\,.
\end{split}
\end{equation}
On the other hand, following Definition \ref{defn_contr_geo} of weighted controlled path to the quantity  $\gp(x_t)=\langle \1^*, \gP(X_t)\rangle$ we have
\begin{equation}\label{equations_quasi3}
\begin{split}
\gp(x_t)-\gp(x_s)&=\sum_{\substack{u\in W(\bA_{\gga}^{d})\\ 1\leq \Vert u\Vert_{\gamma}< N_{\gga}+1}} \langle u^*, \gP(X_s)\rangle\langle \bX_{st}, u\rangle+ o(\abs{t-s})\\&=\sum_{\substack{u\in W(\bA_{\gga}^{d})\\ 1\leq \Vert u\Vert_{\gamma}< N_{\gga}+1}}\frac{\partial^{u}\gp(x_t)}{u!}\langle \bX_{st}, u\rangle+ o(\abs{t-s})\,.
\end{split}
\end{equation}
Subtracting the equation \eqref{equations_quasi2} from  \eqref{equations_quasi3} we obtain
\[ \gp(x_t)-\gp(x_s)- \sum_{n=1}^{N_{\gga}}\sum_{\substack{k \in \bA_{\gga}^d\\ |k|=n}}\frac{1}{k!}\int_s^t\partial^{k}\gP(X_r)d \bX^{k}_r= o(|t-s|)\,.\]
Therefore we can conclude as in the proof of Theorem \ref{chg_thm}.
\end{proof}
The formula \eqref{quasi_shuffle_change} can be obtained  as a trivial consequence of the Theorems \ref{chg_thm} and \ref{big_thm}, provided an explicit identity between branched and geometric rough integrals.
\begin{prop}
Let $x\colon [0,T]\to \bR^{d}$ be a generic $\gga$-H\"older path for any  $\gga\in (0,1)$ and $d\geq 1$. Then for any $\gga$-quasi-geometric rough path $\bX$ over $x$ and all $i_1, \cdots, i_n \in \{1,\cdots, d\}$  we have the identity 
\begin{equation}\label{ident_bra_quasi}
\int_s^t \gP_b(X_r) d^b (\psi_{\cS}^*\bX)^{(i_1\cdots i_n)}_r=\int_s^t \gP_g(X_r) d^g \bX^{[e_{i_1}\cdots e_{i_n}]_{\gga}}_r\,,
\end{equation}
where $\gP_b$ and $\gP_b$ are respectively the lifting of $\gp\in C_b^{N_{\gga}+1-n}(\bR^d,\bR)$  composed with the weighted controlled rough paths \eqref{coordinate2} and \eqref{coordinate3}. Moreover, we adopted the convention $(i_1\cdots i_n)=i_1$ and $[e_{i_1}\cdots e_{i_n}]_{\gga}=e_{i_1}$ when $n=1$.
\end{prop}
\begin{proof}
Accordingly to Proposition \ref{rough_integration} we have the identity modulo $o(|t-s|)$
\[\begin{split}
\int_s^t \gP_b(X_r)  d (\psi_{\cS}^*\bX)^{(i_1\cdots i_n)}_r&= \sum_{m=0}^{N_{\gga}-n}\frac{1}{m!}\sum_{j_1,\cdots, j_{m}=1} ^{d} \partial_{j_1}\cdots \partial_{j_n }\gp(x_s)\langle \psi_{\cS}^*\bX_{st}, B^{+}_{(i_1\cdots i_n)}(\bullet_{j_1}\cdots \bullet_{j_{m}}) \rangle\,,\\\int_s^t \gP_g(X_r) d^g \bX^{[e_{i_1}\cdots e_{i_n}]_{\gga}}_r&= \sum_{\substack{u\in W(\bA_{\gga}^{d})\\  \Vert u\Vert_{\gamma}< N_{\gga}+1-n}}\frac{\partial^{u}\gp(x_s)}{u!}\langle \bX_{st}, u [e_{i_1}\cdots e_{i_n}]_{\gga}\rangle\,.
\end{split}\]
By definition of $\psi_{\cS}$,  for any  $j_1\,,\cdots \,, j_m$ one has
\[\langle \psi_{\cS}^*\bX_{st}, B^{+}_{(i_1\cdots i_n)}(\bullet_{j_1}\cdots \bullet_{j_{m}}) \rangle= \langle\bX_{st},(e_{j_1} \qshuffle_{\gga} \cdots \qshuffle_{\gga}\; e_{j_{m}} ) [e_{i_1}\cdots e_{i_n}]_{\gga}\rangle\,.\]
Using the combinatorial identities \eqref{identity_sum} and  \eqref{comb_lemma} we can write 
\begin{equation}\label{last_equality}
\begin{split}
&\sum_{m=0}^{N_{\gga}-n}\frac{1}{m!}\sum_{j_1,\cdots, j_{m}=1} ^{d} \partial_{j_1}\cdots \partial_{j_m}\gp(x_s)\langle\bX_{st},(e_{j_1} \qshuffle_{\gga} \cdots \qshuffle_{\gga}\; e_{j_{m}} ) [e_{i_1}\cdots e_{i_n}]_{\gga}\rangle\\&=\sum_{0\leq |k|< N_{\gga}+1-n}\frac{\partial^k\gp(x_s)}{k_1!\cdots k_{d}!}\langle\bX_{st},(e_1^{\qshuffle_{\gga} k_1}\qshuffle_{\gga}\cdots \qshuffle_{\gga} e_d^{\qshuffle_{\gga} k_{d}}) [e_{i_1}\cdots e_{i_n}]_{\gga}\rangle\\&=\sum_{\substack{u\in W(\bA_{\gga}^{d})\\  \Vert u\Vert_{\gamma}< N_{\gga}+1-n}}\frac{\partial^{u}\gp(x_s)}{u!}\langle\bX_{st},u[e_{i_1}\cdots e_{i_n}]_{\gga}\rangle\,.
\end{split}
\end{equation}
Thereby obtaining that the difference between the two terms in in \eqref{ident_bra_quasi} is of order  $o(|t-s|)$. The identity \eqref{ident_bra_quasi} comes as  trivial consequence.
\end{proof}
\begin{proof}[An alternative proof of Theorem \ref{quasi_shuffle_change_thm}] Thanks to the Theorem  \ref{big_thm}, the rough path  $\psi_{\cS}^*\bX$ is a simple bracket extension over $x$. Thus the formula \eqref{change_of_variable} reads as
\[\begin{split}
&\gp(x_t)=\\&\gp(x_s)+ \sum_{i=1}^d\int_s^t \partial_i\gP(X_r) d^b(\psi_{\cS}^*\bX)^i_r+ \sum_{n=2}^{N_{\gamma}} \frac{1}{n!} \sum_{i_1,\cdots, i_n=1}^d\int_s^t\partial_{i_1}\cdots\partial_{i_n}\gP(X_r)d^b(\psi_{\cS}^*\bX)^{(i_1 \cdots i_n)}_r\,,
\end{split}
\]
Using the equalities \eqref{ident_bra_quasi} and  \eqref{identity_sum}, we conclude immediately.
\end{proof}

\begin{rk}\label{N=2}
Accordingly to Proposition \eqref{prop_trivial_bracket}, in case $\gga\in (1/3,1/2) $ there exists a unique notion of simple bracket extension over a branched rough paths which is associated to an elementary rough path. Therefore both formulae \eqref{quasi_shuffle_change} and \eqref{change_of_variable} collapse to a unique one which is exactly the one obtained in \cite[Prop. 5.6]{Friz2020course} by means of the notion of \emph{reduced rough paths}. 
\end{rk}

\bibliographystyle{abbrv}
\bibliography{bibliography}

\begin{thebibliography}{10}

\bibitem{Arous1989}
G.~B. Arous.
\newblock Flots et series de {T}aylor stochastiques.
\newblock {\em Probability Theory and Related Fields}, 81(1):29--77, 1989.

\bibitem{bellingeri2020transport}
C.~Bellingeri, A.~Djurdjevac, P.~K. Friz, and N.~Tapia.
\newblock Transport and continuity equations with (very) rough noise.
\newblock {\em Arxiv preprint 2002.10432}, pages 1--20, 2020.

\bibitem{curry2020}
Y.~Bruned, C.~Curry, and K.~Ebrahimi-Fard.
\newblock Quasi-shuffle algebras and renormalisation of rough differential
  equations.
\newblock {\em Bulletin of the London Mathematical Society}, 52(1):43--63,
  2020.

\bibitem{butcher72}
J.~C. Butcher.
\newblock An algebraic theory of integration methods.
\newblock {\em Math. Comp.}, 26:79--106, 1972.

\bibitem{Cartier88}
P.~Cartier.
\newblock Jacobiennes g\'{e}n\'{e}ralis\'{e}es, monodromie unipotente et
  int\'{e}grales it\'{e}r\'{e}es.
\newblock {\em S\'{e}minaire Bourbaki, Vol. 1987/88}, 161-162:31--52, 1988.

\bibitem{Chen54}
K.-T. Chen.
\newblock Iterated integrals and exponential homomorphisms.
\newblock {\em Proc. London Math. Soc. (3)}, 4:502--512, 1954.

\bibitem{Connes1998}
A.~Connes and D.~Kreimer.
\newblock Hopf algebras, renormalization and noncommutative geometry.
\newblock {\em Communications in Mathematical Physics}, 199(1):203--242, Dec
  1998.

\bibitem{Cont2019}
R.~Cont and N.~Perkowski.
\newblock Pathwise integration and change of variable formulas for continuous
  paths with arbitrary regularity.
\newblock {\em Transactions of the American Mathematical Society. Series B},
  6:161--186, 2019.

\bibitem{curry14}
C.~Curry, K.~Ebrahimi-Fard, S.~J. Malham, and A.~Wiese.
\newblock Lévy processes and quasi-shuffle algebras.
\newblock {\em Stochastics}, 86(4):632--642, 2014.

\bibitem{Manchon2018}
C.~Curry, K.~Ebrahimi-Fard, D.~Manchon, and H.~Z. Munthe-Kaas.
\newblock Planarly branched rough paths and rough differential equations on
  homogeneous spaces.
\newblock {\em J. Differential Equations}, 269(11):9740--9782, 2020.

\bibitem{sorindascalescu2000}
S.~D\u{a}sc\u{a}lescu, C.~N\u{a}st\u{a}sescu, and c.~Raianu.
\newblock {\em Hopf algebras}, volume 235 of {\em Monographs and Textbooks in
  Pure and Applied Mathematics}.
\newblock Marcel Dekker, Inc., New York, 2001.

\bibitem{kurusch15}
K.~Ebrahimi-Fard, S.~Malham, F.~Patras, and A.~Wiese.
\newblock The exponential {L}ie series for continuous semimartingales.
\newblock {\em Proceedings of the Royal Society A: Mathematical, Physical and
  Engineering Science}, 471, 06 2015.

\bibitem{Ebrahimi-Fard2015}
K.~Ebrahimi-Fard, S.~J. Malham, F.~Patras, and A.~Wiese.
\newblock Flows and stochastic {T}aylor series in {I}t{\^o} calculus.
\newblock {\em Journal of Physics A: Mathematical and Theoretical},
  48(49):495202, 2015.

\bibitem{fauvet17}
F.~Fauvet and F.~Menous.
\newblock Ecalle's arborification coarborification transforms and
  {C}onnes-{K}reimer {H}opf algebra.
\newblock {\em Annales scientifiques de l'ENS}, 50(1):39 -- 83, 2017.

\bibitem{follmer81}
H.~F{\"o}llmer.
\newblock Calcul d'{I}to sans probabilit\'es.
\newblock {\em S\'eminaire de probabilit\'es de Strasbourg}, 15:143--150, 1981.

\bibitem{Friz2020course}
P.~K. Friz and M.~Hairer.
\newblock {\em A Course on Rough Paths}.
\newblock Springer International Publishing, 2020.

\bibitem{gaines94}
J.~G. Gaines.
\newblock The algebra of iterated stochastic integrals.
\newblock {\em Stochastics and Stochastic Reports}, 49(3-4):169--179, 1994.

\bibitem{Gubinelli2004}
M.~Gubinelli.
\newblock Controlling rough paths.
\newblock {\em Journal of Functional Analysis}, 216(1):86--140, 2004.

\bibitem{gub10}
M.~Gubinelli.
\newblock Ramification of rough paths.
\newblock {\em Journal of Differential Equations}, 248(4):693 -- 721, 2010.

\bibitem{Gyurko2016}
L.~G. Gyurk{\'o}.
\newblock Differential equations driven by $\pi$-rough paths.
\newblock {\em Proceedings of the Edinburgh Mathematical Society},
  59(3):741--758, 2016.

\bibitem{Hairer2014}
M.~Hairer.
\newblock A theory of regularity structures.
\newblock {\em Inventiones mathematicae}, 198(2):269--504, 2014.

\bibitem{hairer2015geometric}
M.~Hairer and D.~Kelly.
\newblock Geometric versus non-geometric rough paths.
\newblock {\em Ann. Inst. H. Poincaré Probabilité et Statistique},
  51(1):207--251, 02 2015.

\bibitem{hoffman2000}
M.~E. Hoffman.
\newblock Quasi-{S}huffle {P}roducts.
\newblock {\em Journal of Algebraic Combinatorics}, 11(1):49--68, Jan 2000.

\bibitem{hoffman17}
M.~E. Hoffman and K.~Ihara.
\newblock Quasi-shuffle products revisited.
\newblock {\em Journal of Algebra}, 481:293 -- 326, 2017.

\bibitem{Hu1988}
Y.~Hu and P.~A. Meyer.
\newblock {\em Sur les integrales multiples de Stratonovitch}, pages 72--81.
\newblock Springer Berlin Heidelberg, Berlin, Heidelberg, 1988.

\bibitem{kelly2012ito}
D.~Kelly.
\newblock {\em Itô corrections in stochastic equations}.
\newblock PhD thesis, University of Warwick, 2012.

\bibitem{Lyons2007}
T.~Lyons and N.~Victoir.
\newblock An extension theorem to rough paths.
\newblock {\em Annales de l'Institut Henri Poincar\'{e}. Analyse Non
  Lin\'{e}aire}, 24(5):835--847, 2007.

\bibitem{lyons1998}
T.~J. Lyons.
\newblock Differential equations driven by rough signals.
\newblock {\em Revista Matemática Iberoamericana}, 14(2):215--310, 1998.

\bibitem{Manchon2008}
D.~Manchon.
\newblock {\em Hopf Algebras in Renormalisation}, volume~5 of {\em Handbook of
  Algebra}.
\newblock North-Holland, 2008.

\bibitem{nourdin_peccati_2012}
I.~Nourdin and G.~Peccati.
\newblock {\em Normal Approximations with Malliavin Calculus: From Stein's
  Method to Universality}.
\newblock Cambridge Tracts in Mathematics. Cambridge University Press, 2012.

\bibitem{reutenauer1993free}
C.~Reutenauer.
\newblock {\em Free Lie Algebras}.
\newblock LMS monographs. Clarendon Press, 1993.

\bibitem{revuz2004continuous}
D.~Revuz and M.~Yor.
\newblock {\em Continuous {M}artingales and {B}rownian {M}otion}.
\newblock Grundlehren der mathematischen Wissenschaften. Springer Berlin
  Heidelberg, 2004.

\bibitem{Rota_Wallstrom97}
G.-C. Rota and T.~C. Wallstrom.
\newblock Stochastic integrals: a combinatorial approach.
\newblock {\em Ann. Probab.}, 25(3):1257--1283, 1997.

\bibitem{Tapia2018}
N.~Tapia and L.~Zambotti.
\newblock The geometry of the space of branched rough paths.
\newblock {\em Proceedings of the London Mathematical Society},
  121(2):220--251, 2020.

\bibitem{Young1936}
W.~H. Young.
\newblock An inequality of the {H}\"older type, connected with {S}tieltjes
  integration.
\newblock {\em Acta Mathematica}, 67(1):251--282, 1936.

\end{thebibliography}

\end{document}